\documentclass[11pt]{article}

\usepackage[a4paper,margin=1in]{geometry}
\usepackage{amsmath,amssymb,amsthm,mathtools,bm}
\usepackage{mathrsfs}
\usepackage{xcolor}
\definecolor{linkblue}{RGB}{0,70,255}
\definecolor{citered}{RGB}{200,0,0}
\definecolor{urlblue}{RGB}{0,120,255}
\usepackage[
  colorlinks=true,
  linkcolor=linkblue,
  citecolor=citered,
  urlcolor=urlblue
]{hyperref}
\usepackage{cases}
\usepackage{empheq}
\usepackage{graphicx}
\usepackage{epstopdf}
\usepackage{amsthm}
\usepackage{multirow}   
\usepackage{amsmath}    
\usepackage{booktabs}

\theoremstyle{plain}

\numberwithin{equation}{section}

\newtheorem{proposition}{Proposition}[section]

\newtheorem{corollary}{Corollary}[section]

\begin{document}

\title{An Onsager Variational Scheme for Pressure-Driven Tumor Growth and Hele–Shaw Limits}

\author{
Weijie Huang\thanks{School of Mathematics and Statistics, Beijing Jiaotong University, Beijing 100044, People's Republic of China; Beijing Key Laboratory of Biological Big Data and Topological Statistics, Beijing Jiaotong University, Beijing 100044, People's Republic of China  {\tt (wjhuang@bjtu.edu.cn)}.} \ and
Xinran Ruan\thanks{(Corresponding author) School of Mathematical Sciences, Capital Normal University, Beijing 100048, People's Republic of China {\tt (xinran.ruan@cnu.edu.cn)}.}
}

\date{}
\maketitle

\begin{abstract}
Pressure-driven tumor growth models describe the coupling between cell proliferation and mechanical pressure and naturally lead to moving free boundary problems. Their numerical approximation is challenging due to degenerate diffusion, pressure-dependent proliferation, and the stiffness of the pressure law \(p=n^\gamma\) for large \(\gamma\). In this paper, we propose a structure-preserving finite difference method for this class of pressure-driven tumor growth models with pressure-dependent proliferation. The method is derived from the Onsager variational principle. The key idea is to introduce a modified energy shifted by the homeostatic pressure, so that the growth term can be written in a dissipative form and incorporated together with the transport part into a unified Rayleighian formulation. This formulation leads to a time-discrete constrained minimization problem and a fully discrete scheme with explicit mobilities and an implicit pressure update. We prove that the scheme preserves nonnegativity and the homeostatic upper bound, satisfies a discrete modified energy dissipation law, and admits a fixed-grid stiff-pressure limiting structure. Numerical experiments in one and two spatial dimensions demonstrate the accuracy of the method, its convergence toward the Hele--Shaw limit for large \(\gamma\), and its ability to capture free boundary evolutions with topology changes.

\end{abstract}

{\bf Key words.} Pressure-driven tumor growth, Onsager variational principle, structure-preserving scheme, energy stability, Hele--Shaw limit.

{\bf MSC 2020. } 65M06, 65M12, 35K65, 35R35, 92C50.

\section{Introduction}

Living tissues are active materials in which proliferation, mechanical
pressure, and crowding effects are strongly coupled
\cite{Basan2009,GomezGonzalez2020,Shraiman2005}. Continuum models for such
mechanobiological processes must describe both bulk evolution and the motion
of free boundaries separating occupied and empty regions, and, in the
stiff-pressure regime, they naturally lead to singular limits in which strong
mechanical resistance enforces an incompressibility constraint in densely
packed tissues
\cite{PerthameQuirosTangVauchelet2014,PerthameQuirosVazquez2014,PerthameVauchelet2015}.
Tumor growth is a prototypical setting in which these mechanisms appear. For
this reason, pressure-driven tumor growth models have become a useful class
of problems for studying the connection between degenerate diffusion, free
boundary motion, and reliable numerical simulation.

In this paper, we consider the pressure-driven tumor growth model
\cite{PerthameQuirosVazquez2014}
\begin{equation}\label{eq:tumor-model}
\partial_t n+\nabla\cdot j=nG(p),\qquad
j=-n\nabla p,\qquad
p=n^\gamma ,
\end{equation}
posed in a bounded domain \(\Omega\subset\mathbb R^d\). Here
\(n=n(\mathbf{x},t)\ge0\) denotes the tumor cell density,
\(j=j(\mathbf{x},t)\) is the cell flux, \(p=p(\mathbf{x},t)\) is the
mechanical pressure, and \(\gamma>1\) is the stiffness parameter in the
pressure law. The function \(G\) is the net proliferation rate. We assume
pressure-inhibited proliferation with a homeostatic pressure \(p_H>0\)
\cite{Basan2009,PerthameQuirosTangVauchelet2014,PerthameQuirosVazquez2014},
namely
\begin{equation}\label{eq:G_p_cond}
G(p_H)=0,\qquad G'(p)\le 0,\quad p\ge0 .
\end{equation}
This implies
\begin{equation}\label{eq:G_p_ineq}
(p-p_H)G(p)\le 0,\qquad p\ge0,
\end{equation}
which is the key structural condition behind the dissipative formulation
used in this work.

For fixed \(\gamma\), \eqref{eq:tumor-model} is a degenerate
reaction-diffusion equation of porous-medium type. The degeneracy gives
finite speed of propagation and allows the formation of sharp moving
interfaces. As \(\gamma\to\infty\), the pressure law enforces a congestion constraint and
the pressure becomes a Lagrange multiplier for the saturated region.
Formally, provided the corresponding bounds hold initially, the limiting
density and pressure satisfy \cite{DavidRuan2022,PerthameQuirosVazquez2014}
\begin{equation}\label{eq:n_p_limit}
0\le n_\infty\le 1,\qquad
0\le p_\infty\le p_H,\qquad
p_\infty(1-n_\infty)=0,
\end{equation}
 together with the complementarity relation
\begin{equation}\label{eq:limiting_p}
p_\infty(\Delta p_\infty+G(p_\infty))=0 .
\end{equation}
In the patch case
\(n_\infty=\mathbf{1}_{\Omega(t)}\), with
\(\Omega(t)=\{p_\infty>0\}\), this gives the Hele--Shaw free boundary
problem
\begin{subequations}\label{eq:HS_limit}
\begin{empheq}[left=\empheqlbrace]{alignat=2}
&-\Delta p_\infty = G(p_\infty), \quad
&& \text{in } \Omega(t), \label{eq:HS_limit_pressure}\\
&p_\infty = 0, \quad
&& \text{on } \partial\Omega(t), \label{eq:HS_limit_boundary}\\
&V_n = -\partial_\nu p_\infty, \quad
&& \text{on } \partial\Omega(t). \label{eq:HS_limit_velocity}
\end{empheq}
\end{subequations}
Here \(V_n\) is the outward normal velocity of the free boundary and
\(\nu\) is the outward unit normal to \(\Omega(t)\).

The mathematical analysis of these models has developed rapidly. B.
Perthame, F. Quir\'os, and J. L. V\'azquez established the Hele--Shaw
asymptotics for mechanical tumor growth models and derived the
complementarity relation in the incompressible limit
\cite{PerthameQuirosVazquez2014}. I. Kim and N. Po\v{z}\'ar proved
convergence from the porous medium equation to Hele--Shaw flow for general
initial densities \cite{KimPozar2018}. B. Perthame and N. Vauchelet studied
the incompressible limit of a mechanical tumor growth model with viscosity
\cite{PerthameVauchelet2015}. B. Perthame, F. Quir\'os, M. Tang, and N.
Vauchelet derived Hele--Shaw-type systems from cell models that include
active motion \cite{PerthameQuirosTangVauchelet2014}. N. David and B.
Perthame analyzed tumor growth models coupled to nutrients, including the
free boundary limit and pressure regularity \cite{DavidPerthame2021}.

From the numerical point of view, \eqref{eq:tumor-model} contains several
interacting difficulties. The diffusion degenerates in vacuum regions, so
standard parabolic smoothing is not available near the interface. The
reaction term changes the mass and must remain compatible with
pressure-inhibited proliferation. For large \(\gamma\), the pressure law is
stiff and the numerical density should approach the congestion constraint
without spurious overshoots. A useful scheme should therefore preserve
nonnegativity, the homeostatic upper bound, and a suitable energy
dissipation law, while also retaining a fixed-grid stiff-pressure limiting
structure in the Hele--Shaw limit.

There is a large literature on numerical methods for porous medium and
degenerate diffusion equations. J. L. Graveleau and P. Jamet developed early
finite difference methods for nonlinear diffusion with finite speed of
propagation \cite{GraveleauJamet1971}. E. DiBenedetto and D. Hoff proposed
an interface tracking algorithm for the porous medium equation
\cite{DiBenedettoHoff1984}. L. Monsaingeon designed an explicit finite
difference method for one-dimensional generalized porous medium equations,
with particular attention to interface tracking and hole filling
\cite{Monsaingeon2016}. Y. Liu, C.-W. Shu, and M. Zhang constructed
high-order finite difference WENO schemes that reduce oscillations near
sharp interfaces \cite{LiuShuZhang2011}. M. Bessemoulin-Chatard and F.
Filbet developed finite volume schemes for nonlinear degenerate parabolic
equations \cite{BessemoulinFilbet2012}. J. A. Carrillo, H. Ranetbauer, and
M.-T. Wolfram used evolving diffeomorphisms to approximate nonlinear
continuity equations \cite{CarrilloRanetbauerWolfram2016}. C. Liu and Y.
Wang constructed Lagrangian schemes from a discrete energetic variational
approach \cite{LiuWang2020}.

Numerical methods have also been developed specifically for tumor growth
models and their Hele--Shaw limits. J.-G. Liu, M. Tang, L. Wang, and Z.
Zhou proposed an accurate front capturing scheme for tumor growth models
with a free boundary limit \cite{LiuTangWangZhou2018}. The same authors studied tumor growth models with nutrients,
connecting cell-density models to free-boundary dynamics \cite{LiuTangWangZhou2019}. N. David and
X. Ruan studied an upwind finite difference scheme, proved stability
estimates and an asymptotic-preserving property, and investigated focusing
solutions numerically \cite{DavidRuan2022}. Most existing methods start
from a direct discretization of the PDE. In contrast, the present work
starts from a modified energy and an Onsager Rayleighian, and treats the
transport flux and the pressure-dependent growth term in the same
variational framework.

The Onsager variational principle provides a systematic way to derive
dissipative dynamics from an energy and a dissipation mechanism. L. Onsager
introduced the reciprocal relations and the variational structure of
irreversible processes \cite{Onsager1931a,Onsager1931b}. M. Doi developed and promoted 
the Onsager principle as a modeling and approximation tool for soft matter
systems \cite{Doi2013,Doi2015}. R. Jordan, D. Kinderlehrer, and F. Otto
showed that the Fokker--Planck equation can be obtained as a minimizing
movement scheme in the Wasserstein metric \cite{JordanKinderlehrerOtto1998}.
C. Duan, C. Liu, C. Wang, and X. Yue used an energetic variational approach
for the porous medium equation \cite{DuanLiuWangYue2019}. H. Chen, H. Liu,
and X. Xu recently presented a general Onsager-based framework for
structure-preserving numerical schemes, emphasizing energy stability and
constraints through discrete Rayleighian minimization \cite{ChenLiuXu2025}.

The key observation of this paper is that the correct dissipative variable for
pressure-inhibited proliferation is obtained by shifting the pressure energy by
the homeostatic pressure. With this shift, the variational derivative is
measured relative to $p_H$, and the pressure-inhibition condition makes the
reaction contribution compatible with energy dissipation. This allows the
transport flux and the reaction source to be incorporated into a single
Rayleighian formulation. Based on this structure, we derive an Onsager
variational formulation for pressure-driven tumor growth with
pressure-inhibited proliferation, construct a time-discrete scheme from a
discrete Rayleighian, and prove its modified energy stability. We then propose
a fully discrete finite difference scheme with explicit mobilities and an
implicit pressure update, and show that it preserves nonnegativity, the
homeostatic upper bound, and a discrete modified energy dissipation law. We
also establish a fixed-grid stiff-pressure limiting structure as
$\gamma\to\infty$, including a discrete density-pressure complementarity
relation and a Hele--Shaw-type pressure equation in saturated cells. Finally, one- and two-dimensional numerical experiments are presented to assess the accuracy, structure preservation, large-$\gamma$ behavior, and ability of
the method to capture free boundary evolutions with topology changes.

The rest of the paper is organized as follows. In Section~\ref{sec:model},
we present the model, state the assumptions on the growth function, and
collect several structural properties. In Section~\ref{sec:onsager},  we develop the Onsager-based numerical scheme and analyze its
structure-preserving properties and fixed-grid stiff-pressure limiting structure. Numerical experiments are reported in
Section~\ref{sec:numerics}. Finally, conclusions are given in
Section~\ref{sec:conclusion}.

\section{Basic properties of the model \eqref{eq:tumor-model}}\label{sec:model}
In this section, we recall the structural properties of the continuous model.
Throughout this
section, \(n=n(\mathbf{x},t)\), \(j=j(\mathbf{x},t)\), and \(p=p(\mathbf{x},t)\) are as in
\eqref{eq:tumor-model}, and the equations are considered in
\(\Omega\times(0,T]\). 
We use the initial condition
$
    n(\mathbf{x},0)=n_0(\mathbf{x}), \,  \mathbf{x}\in\Omega ,
$
and the no-flux boundary condition
\begin{equation}\label{eq:no-flux}
    j\cdot\nu=0,
    \qquad \mathbf{x}\in\partial\Omega,\quad 0<t\le T .
\end{equation}
Here \(\nu\) is the outward unit normal vector of $\partial \Omega$.
For the growth function \(G\) in \eqref{eq:tumor-model}, we assume that
it satisfies \eqref{eq:G_p_cond}-\eqref{eq:G_p_ineq}.
With this assumption, the following elementary estimates hold for smooth
solutions.
\begin{proposition}[Positivity, upper bound, and mass estimate]
\label{prop:continuous-bounds}
Assume that
\[
    0\le p_0(\mathbf{x})=n_0(\mathbf{x})^\gamma\le p_H
    \qquad \text{in } \Omega .
\]
Then any smooth solution of \eqref{eq:tumor-model} subject to
 \eqref{eq:no-flux} satisfies
\[
    0\le p(\mathbf{x},t)\le p_H,
    \qquad
    0\le n(\mathbf{x},t)\le n_H:=p_H^{1/\gamma},
    \qquad \mathbf{x}\in\Omega,\quad 0\le t\le T .
\]
Moreover,
\begin{equation}\label{eq:l1-bound-property}
    \int_\Omega n(\mathbf{x},t)\,d\mathbf{x}
    \le
    e^{G(0)t}\int_\Omega n_0(\mathbf{x})\,d\mathbf{x},
    \qquad 0\le t\le T .
\end{equation}
\end{proposition}

\begin{proof}
Differentiating \(p=n^\gamma\) and using \eqref{eq:tumor-model}, we obtain
\begin{equation}\label{eq:pressure-equation-property}
    \partial_t p
    =
    |\nabla p|^2+\gamma p\bigl(\Delta p+G(p)\bigr).
\end{equation}
The constants \(0\) and \(p_H\) are respectively a subsolution and a
supersolution of this equation, since \(G(p_H)=0\). Hence the comparison
principle gives \(0\le p\le p_H\). The bound
\(0\le n\le n_H\) follows from \(p=n^\gamma\).

For the mass estimate, integrating the density equation and using the
no-flux boundary condition gives
\[
    \frac{d}{dt}\int_\Omega n\,d\mathbf{x}
    =
    \int_\Omega nG(p)\,d\mathbf{x}
    \le
    G(0)\int_\Omega n\,d\mathbf{x},
\]
where we used \(p\ge0\), \(n\ge0\), and the monotonicity of \(G\).
Gronwall's inequality then yields \eqref{eq:l1-bound-property}.
\end{proof}

\begin{proposition}[Modified energy dissipation]
\label{prop:energy-dissipation}
Define the modified energy by
\begin{equation}\label{eq:modified-energy}
    \mathcal E(n)
    =
    \int_\Omega
    \left(
    \frac{1}{\gamma+1}n^{\gamma+1}-p_H n
    \right)\,d\mathbf{x} .
\end{equation}
Let \(n\) be a nonnegative smooth solution of \eqref{eq:tumor-model}
subject to the no-flux boundary condition \eqref{eq:no-flux}. Then
\begin{equation}\label{eq:energy-dissipation}
    \frac{d}{dt}\mathcal E(n)
    =
    -\int_\Omega n|\nabla p|^2\,d\mathbf{x}
    +
    \int_\Omega n(p-p_H)G(p)\,d\mathbf{x}
    \le 0 .
\end{equation}
\end{proposition}

\begin{proof}
Since
\[
    \frac{\delta\mathcal E}{\delta n}=n^\gamma-p_H=p-p_H,
\]
we multiply the density equation by \(p-p_H\) and use the no-flux boundary
condition to obtain
\[
\begin{aligned}
    \frac{d}{dt}\mathcal E(n)
    &=
    \int_\Omega (p-p_H)\partial_t n\,d\mathbf{x}  \\
    &=
    -\int_\Omega n|\nabla p|^2\,d\mathbf{x}
    +
    \int_\Omega n(p-p_H)G(p)\,d\mathbf{x} .
\end{aligned}
\]
The boundary term vanishes because \(n\nabla p\cdot\nu=-j\cdot\nu=0\) on
\(\partial\Omega\). Since \(n\ge0\) and \((p-p_H)G(p)\le0\) by \eqref{eq:G_p_ineq}, the right-hand
side is nonpositive.
\end{proof}

The linear shift \(-p_H n\) is introduced so that the variational derivative
of the energy becomes \(p-p_H\). This makes the reaction contribution
dissipative, since \(n(p-p_H)G(p)\le0\), and will be used in the Onsager
formulation below.


\section{Onsager formulation and numerical scheme}
\label{sec:onsager-scheme}
\label{sec:onsager}

We now derive the numerical scheme from the Onsager variational principle.
For a given density \(n\), the instantaneous evolution is described
by the density rate \(\partial_t n\), the flux \(j\), and the reaction rate
\(r\). To include transport and reaction in the same framework, we write the
density equation as the balance law
\begin{equation}\label{eq:balance-with-r}
    \partial_t n+\nabla\cdot j=r .
\end{equation}
The Onsager principle determines \(\partial_t n\), \(j\), and \(r\) by
minimizing the Rayleighian
\[
    \mathcal R=\Phi+\frac{d}{dt}\mathcal E
\]
under the constraint \eqref{eq:balance-with-r}. Here \(\Phi\) is a
nonnegative quadratic dissipation functional, and
\(\frac{d}{dt}\mathcal E\) is the instantaneous variation of the modified
energy \(\mathcal E\) defined in \eqref{eq:modified-energy}. The
Euler--Lagrange equations of this constrained minimization will recover the
constitutive relations
\[
    j=-n\nabla p,
    \qquad
    r=nG(p).
\]
Substituting these relations into \eqref{eq:balance-with-r}, together with
\(p=n^\gamma\), gives the original model \eqref{eq:tumor-model}.

\paragraph{Choice of the Rayleighian.}
We now choose the dissipation functional \(\Phi\) so that the above
constrained minimization gives the desired constitutive relations. For
notational simplicity, we introduce
\begin{equation}\label{eq:modified-chemical-potential}
    \mu
    :=
    \frac{\delta\mathcal E}{\delta n}
    =
    p-p_H .
\end{equation}
In the continuous Onsager formulation, \(n\) is regarded as the current
state and is not varied. Hence \(p=n^\gamma\), \(\mu\), and the mobilities
defined below are fixed coefficients in the minimization, while
\(\partial_t n\), \(j\), and \(r\) are the variables.

For the transport part, we use the standard quadratic dissipation
\[
    \frac12\int_\Omega \frac{|j|^2}{n}\,d\mathbf{x} .
\]
When the Rayleighian is minimized subject to the balance law
\eqref{eq:balance-with-r}, this term leads to the flux relation
\[
    j=-n\nabla\mu .
\]
Since \(\nabla\mu=\nabla p\), the desired flux law
\(j=-n\nabla p\) is recovered.

The main point is the treatment of the reaction term. We use a quadratic
reaction dissipation of the form
\[
    \frac12\int_\Omega \frac{r^2}{\mathcal M(n,p)}\,d\mathbf{x} .
\]
In the same Rayleighian minimization, this term leads to the Onsager
relation
\[
    r=-\mathcal M(n,p)\mu .
\]
To recover the growth term \(r=nG(p)\), and using \(\mu=p-p_H\), we first
define the continuous extension of the secant slope
\begin{equation}\label{eq:g-secant}
    g(p)
    =
    \begin{cases}
    \dfrac{G(p)}{p-p_H}, & p\ne p_H,\\[1.2ex]
    G'(p_H), & p=p_H .
    \end{cases}
\end{equation}
Then the reaction mobility is defined by
\begin{equation}\label{eq:reaction-mobility}
    \mathcal M(n,p)
    =
    -n g(p).
\end{equation}
By the assumptions on \(G\) \eqref{eq:G_p_cond}-\eqref{eq:G_p_ineq}, \(g\) is continuous and \(g(p)\le0\) for
\(p\ge0\). Hence
\[
    \mathcal M(n,p)\ge0
    \qquad \text{for } n\ge0,\quad p\ge0 .
\]
Thus the reaction dissipation is nonnegative and yields the desired growth
term through the Onsager relation.

Combining the transport and reaction parts, we define
\begin{equation}\label{eq:total-dissipation}
    \Phi(n;j,r)
    =
    \frac12\int_\Omega \frac{|j|^2}{n}\,d\mathbf{x}
    +
    \frac12\int_\Omega \frac{r^2}{\mathcal M(n,p)}\,d\mathbf{x} .
\end{equation}
The dissipation terms in \eqref{eq:total-dissipation} should be understood in
the extended sense. Namely, for \(a\ge0\),
\[
    \frac{|z|^2}{a}
    :=
    \begin{cases}
    |z|^2/a, & a>0,\\
    0, & a=0,\ z=0,\\
    +\infty, & a=0,\ z\ne0 .
    \end{cases}
\]
With these notations, the corresponding Rayleighian is
\begin{equation}\label{eq:rayleighian}
    \mathcal R(n;\partial_t n,j,r)
    =
    \Phi(n;j,r)
    +
    \int_\Omega \mu\,\partial_t n\,d\mathbf{x} .
\end{equation}
Here the semicolon emphasizes that \(n\) is the given state, while
\(\partial_t n\), \(j\), and \(r\) are the variables in the minimization.

\paragraph{Consistency verification.}
For completeness, we indicate the Euler--Lagrange relations. Introducing a
Lagrange multiplier \(\lambda\) for the constraint
\eqref{eq:balance-with-r}, we define the augmented Rayleighian
\[
    \mathcal R_\lambda
    =
    \mathcal R
    -
    \int_\Omega
    \lambda\left(\partial_t n+\nabla\cdot j-r\right)\,d\mathbf{x} .
\]
Taking the first variations of \(\mathcal R_\lambda\) with respect to
\(\partial_t n\), \(j\), and \(r\) gives
\[
    \lambda=\mu,
    \qquad
    \frac{j}{n}+\nabla\lambda=0,
    \qquad
    \frac{r}{\mathcal M(n,p)}+\lambda=0 .
\]
Hence
\begin{equation}\label{eq:onsager-el-relations}
    j=-n\nabla\mu,
    \qquad
    r=-\mathcal M(n,p)\mu .
\end{equation}
Using the definition of $\mu$ \eqref{eq:modified-chemical-potential} and \(\mathcal M(n,p)\) \eqref{eq:g-secant}-\eqref{eq:reaction-mobility}, we recover
\[
    j=-n\nabla p,
    \qquad
    r=nG(p).
\]
Thus the Onsager formulation is consistent with \eqref{eq:tumor-model}.

\subsection{Time-discrete scheme from a discrete Rayleighian}
\label{subsec:time-discrete}

We next derive a time-discrete scheme from the Rayleighian formulation and
record its bound-preserving and energy-dissipating properties.

Let \(\tau>0\) be the time step and \(t_k=k\tau\). Suppose that \(n^k\) is
known and set
\[
    p^k=(n^k)^\gamma .
\]
The mobilities in the dissipation terms are evaluated explicitly at the
known state. Thus the transport mobility is \(n^k\), and the reaction
mobility is
\begin{equation}\label{eq:discrete-reaction-mobility}
    \mathcal M^k
    =
    -n^k g(p^k).
\end{equation}
By \(g(p^k)\le0\) for \(p^k\ge0\),
\begin{equation} \label{prop:M}
    \mathcal M^k\ge0
    \qquad \text{if } n^k\ge0 .
\end{equation}

We denote by \(m\) and \(q\) the time-integrated flux and reaction over one
time step. The discrete balance law of \eqref{eq:balance-with-r} is
\begin{equation}\label{eq:discrete-balance}
    n-n^k+\nabla\cdot m=q,
    \qquad
    m\cdot\nu=0
    \quad \text{on }\partial\Omega .
\end{equation}
Here the trial density \(n\) represents the new-time value, \(m=\tau j\) and \(q=\tau r\). 
Motivated by the
continuous Rayleighian, we define the discrete Rayleighian
\begin{equation}\label{eq:discrete-rayleighian}
    \mathcal J_\tau^k(n,m,q)
    =
    \mathcal E(n) - \mathcal E(n^k)
    +
    \frac{1}{2\tau}\int_\Omega \frac{|m|^2}{n^k}\,d\mathbf{x}
    +
    \frac{1}{2\tau}\int_\Omega \frac{q^2}{\mathcal M^k}\,d\mathbf{x} .
\end{equation}
This functional is obtained by replacing \(d\mathcal E/dt\) with the energy
increment over one time step, freezing the mobilities at \(n^k\).

The next time step is defined by the constrained minimization
\begin{equation}\label{eq:discrete-minimization}
    (n^{k+1},m^{k+1},q^{k+1})
    =
    \operatorname*{arg\,min}_{(n,m,q)\in\mathcal A^k}
    \mathcal J_\tau^k(n,m,q),
\end{equation}
where
\begin{equation}\label{eq:admissible-set}
    \mathcal A^k
    =
    \left\{
    (n,m,q)\ \middle|\
    n\ge0,\quad
    n-n^k+\nabla\cdot m=q \ \text{in }\Omega,\quad
    m\cdot\nu=0 \ \text{on }\partial\Omega
    \right\}.
\end{equation}

Since the Rayleighian \eqref{eq:discrete-rayleighian} is a sum of separate convex terms in \(n\), \(m\), and
\(q\), it is convex in the joint variable \((n,m,q)\). 
Together with the convexity of the admissible set \(\mathcal A^k\),  this implies that the minimization problem \eqref{eq:discrete-minimization} is convex. 

The minimizer of \eqref{eq:discrete-minimization} can be characterized by
its Euler--Lagrange equations. Introducing a Lagrange multiplier
\(\lambda\) for the discrete balance constraint \eqref{eq:discrete-balance},
we define
\[
    \mathcal J_{\tau,\lambda}^k
    =
    \mathcal J_\tau^k
    -
    \int_\Omega
    \lambda\left(n-n^k+\nabla\cdot m-q\right)\,dx .
\]
Let \((n^{k+1},m^{k+1},q^{k+1})\) be a smooth minimizer, and let
\(\lambda^{k+1}\) be the corresponding multiplier. 
On the positivity set of $n^{k+1}$, taking variations with
respect to \(n\), \(m\), \(q\), and \(\lambda\), we obtain
\begin{align}
    &\frac{\delta\mathcal E}{\delta n}(n^{k+1})
    =\lambda^{k+1},
    \label{eq:el-discrete-n}\\
    &m^{k+1}=-\tau n^k\nabla\lambda^{k+1},
    \label{eq:el-discrete-m}\\
    &q^{k+1}=-\tau\mathcal M^k\lambda^{k+1},
    \label{eq:el-discrete-q}\\
    &n^{k+1}-n^k+\nabla\cdot m^{k+1}=q^{k+1}.
    \label{eq:el-discrete-constraint}
\end{align}
Here the boundary term in the variation with respect to \(m\) vanishes
because of the boundary condition in \eqref{eq:discrete-balance}.

Since
\[
    \lambda^{k+1}
    =
    \mu^{k+1}
    :=
    p^{k+1}-p_H,
    \qquad
    p^{k+1}=(n^{k+1})^\gamma ,
\]
substituting \eqref{eq:el-discrete-m}--\eqref{eq:el-discrete-q} into
\eqref{eq:el-discrete-constraint} gives
\begin{equation}\label{eq:time-discrete-scheme}
    \frac{n^{k+1}-n^k}{\tau}
    =
    \nabla\cdot\left(n^k\nabla\mu^{k+1}\right)
    -
    \mathcal M^k\mu^{k+1}.
\end{equation}
Equivalently, using \(\mu^{k+1}=p^{k+1}-p_H\), we obtain the pressure form
\begin{equation}\label{eq:time-discrete-scheme-pressure}
    \frac{n^{k+1}-n^k}{\tau}
    =
    \nabla\cdot\left(n^k\nabla p^{k+1}\right)
    -
    \mathcal M^k\left(p^{k+1}-p_H\right),
    \qquad
    p^{k+1}=(n^{k+1})^\gamma .
\end{equation}
Moreover,  from \(m^{k+1}\cdot\nu=0\) on $\partial \Omega$, we recover the boundary condition 
\begin{equation}\label{eq:time-discrete-boundary}
    n^k\nabla p^{k+1}\cdot\nu=0
    \qquad \text{on }\partial\Omega .
\end{equation}

The resulting scheme can be viewed as a natural time-discrete counterpart of
\eqref{eq:tumor-model}. Its main feature is the explicit evaluation of the
mobilities together with the implicit update of the pressure.
 This explicit--implicit treatment keeps the scheme simple while preserving the
main structural properties of the continuous model. 

Since the minimization problem \eqref{eq:discrete-minimization} is convex, 
the Euler--Lagrange system \eqref{eq:time-discrete-scheme-pressure} derived above
is not only a necessary condition for a smooth minimizer, but also a
sufficient condition for global minimality, whenever the corresponding
solution is admissible. In this sense, the time-discrete scheme may be
equivalently described by the constrained minimization problem \eqref{eq:discrete-minimization} or by its
Euler--Lagrange equations \eqref{eq:time-discrete-scheme-pressure}.

Before passing to the fully discrete finite difference formulation, we show the positivity,
homeostatic upper bound, and energy dissipation of the time-discrete scheme.

\begin{proposition}[Positivity and homeostatic upper bound]
\label{prop:td-positivity-upper-bound}
Assume that \(n^k\ge0\). Let
\((n^{k+1},m^{k+1},q^{k+1})\) be a minimizer of
\eqref{eq:discrete-minimization}. Then
\[
    n^{k+1}\ge0 .
\]
If, in addition, \(n^k\le n_H := p_H^{\frac{1}{\gamma}}\) and the minimizer satisfies
\eqref{eq:time-discrete-scheme-pressure}--\eqref{eq:time-discrete-boundary},
then
\[
    0\le n^{k+1}\le n_H,
    \qquad
    0\le p^{k+1}\le p_H .
\]
\end{proposition}

\begin{proof}
The nonnegativity follows directly from the admissible set
\(\mathcal A^k\), where \(n\ge0\) is imposed.

To prove the upper bound, set
\[
    w=(p^{k+1}-p_H)_+ .
\]
On the set \(\{w>0\}\), we have \(p^{k+1}>p_H\), hence
\(n^{k+1}>n_H\ge n^k\). Therefore
\[
    (n^{k+1}-n^k)w > 0 .
\]
Multiplying \eqref{eq:time-discrete-scheme-pressure} by \(w\), integrating
over \(\Omega\), and using \eqref{eq:time-discrete-boundary}, we obtain
\[
    \frac1\tau\int_\Omega (n^{k+1}-n^k)w\,dx
    =
    -\int_\Omega n^k|\nabla w|^2\,dx
    -
    \int_\Omega \mathcal M^k w^2\,dx
    \le0 .
\]
The left-hand side is nonnegative, and hence it must vanish. 
We conclude that \(w=0\) and thus
\(p^{k+1}\le p_H\). 
Together with \(n^{k+1}\ge0\) and
\(p^{k+1}=(n^{k+1})^\gamma\), this gives
\[
    0\le p^{k+1}\le p_H,
    \qquad
    0\le n^{k+1}\le n_H .
\]
\end{proof}

\begin{proposition}[Energy stability]
\label{prop:time-discrete-energy}
Assume that \(n^k\ge0\). Let
\((n^{k+1},m^{k+1},q^{k+1})\) be a minimizer of
\eqref{eq:discrete-minimization}. Then
\begin{equation}\label{eq:time-discrete-energy}
\mathcal E(n^{k+1})
+
\frac{1}{2\tau}
\int_\Omega
\frac{|m^{k+1}|^2}{n^k}\,dx
+
\frac{1}{2\tau}
\int_\Omega
\frac{|q^{k+1}|^2}{\mathcal M^k}\,dx
\le
\mathcal E(n^k).
\end{equation}
In particular, \(\mathcal E(n^{k+1})\le\mathcal E(n^k)\).
\end{proposition}

\begin{proof}
Since \((n^k,0,0)\in\mathcal A^k\), the minimizing property gives
\[
\mathcal J_\tau^k(n^{k+1},m^{k+1},q^{k+1})
\le
\mathcal J_\tau^k(n^k,0,0)
=
0.
\]
Using the definition of \(\mathcal J_\tau^k\) gives
\eqref{eq:time-discrete-energy}. The last assertion follows from the
nonnegativity of the two dissipation terms.
\end{proof}



\subsection{Fully discrete finite difference scheme}
\label{subsec:fully-discrete}
\label{sec:fully-discrete}

We now present the fully discrete finite difference scheme. For simplicity,
we give the one-dimensional formulation on \(\Omega=(a,b)\).
The two-dimensional extension used in the computations is stated in
Appendix~\ref{app:2d-extension}.

Let \(h=(b-a)/N\), \(x_{i+\frac12}=a+ih\) for \(0\le i\le N\), and
\(x_i=a+(i-\frac12)h\) for \(1\le i\le N\). We denote by \(n_i^k\) the
cell-centered approximation of \(n(x_i,t_k)\), and set
\(
    p_i^k=(n_i^k)^\gamma .
\)
For a cell-centered grid function \(v=\{v_i\}_{i=1}^N\) and an interface
grid function \(F=\{F_{i+\frac12}\}_{i=0}^N\), define
\begin{equation}\label{eq:dv_df}
    (\delta_h v)_{i+\frac12}
    =
    \frac{v_{i+1}-v_i}{h},
    \qquad
    (d_hF)_i
    =
    \frac{F_{i+\frac12}-F_{i-\frac12}}{h}.
\end{equation}

Given \(n^k\), the new density \(n^{k+1}\) is determined as follows. Let
\begin{equation}\label{eq:fd-flux}
    n_{i+\frac12}^k=\frac{n_i^k+n_{i+1}^k}{2}, \quad
    F_{i+\frac12}^{k+1}
    =
    n_{i+\frac12}^k(\delta_h p^{k+1})_{i+\frac12},
    \quad
    1\le i\le N-1.
\end{equation}
Here \(F\) denotes the pressure-gradient flux \(n\nabla p\), that is, the
negative of the physical flux \(j\).
The no-flux boundary condition is imposed by
\begin{equation}\label{eq:fd-no-flux}
    F_{\frac12}^{k+1}=F_{N+\frac12}^{k+1}=0 .
\end{equation}
Then the fully discrete scheme is
\begin{equation}\label{eq:fd-scheme}
    \frac{n_i^{k+1}-n_i^k}{\tau}
    =
    (d_hF^{k+1})_i
    -
    \mathcal M_i^k\left(p_i^{k+1}-p_H\right),
    \qquad 1\le i\le N,
\end{equation}
where
\begin{equation}\label{eq:fd-reaction-mobility}
    \mathcal M_i^k
    =
    -n_i^k g(p_i^k).
\end{equation}
Here \(p_i^{k+1}=(n_i^{k+1})^\gamma\). 
Since \(g(p_i^k)\le0\) for \(p_i^k\ge0\), if \(n_i^k\ge0\), then
\begin{equation} \label{prop:M_discrete}
	\mathcal M_i^k\ge0,\qquad 1\le i\le N. 
\end{equation}
The fully discrete scheme \eqref{eq:fd-scheme} is a nonlinear system for \(n^{k+1}\), which can be solved
by a damped Newton iteration in the computations. See
Appendix~\ref{app:nonlinear-solver}. 
Below is the fully discrete update of $n^k$ in one time step. 

\begin{center}
\fbox{%
\begin{minipage}{0.92\textwidth}
\textbf{Fully discrete update.}
Given \(n^k\),
\begin{enumerate}
\item Set \(p_i^k=(n_i^k)^\gamma\).
\item Compute \(\mathcal M_i^k=-n_i^k g(p_i^k)\).
\item Solve the nonlinear system \eqref{eq:fd-scheme} for
\(n_i^{k+1}\), with \(p_i^{k+1}=(n_i^{k+1})^\gamma\).
\end{enumerate}
\end{minipage}}
\end{center}

\begin{proposition}[Existence and nonnegativity]
\label{prop:fd-existence-nonnegativity}
Assume that \(n_i^k\ge0\) for \(1\le i\le N\). Then the fully discrete
scheme \eqref{eq:fd-flux}--\eqref{eq:fd-reaction-mobility} admits a
nonnegative solution
\[
    n_i^{k+1}\ge0,\qquad 1\le i\le N .
\]
\end{proposition}

\begin{proof}
We use the pressure variable. For \(p=\{p_i\}_{i=1}^N\), consider the functional
\[
\begin{aligned}
\mathcal J_h(p)
&=
h\sum_{i=1}^N
\left[
\frac{\gamma}{\gamma+1}p_i^{\frac{\gamma+1}{\gamma}}
-
n_i^k p_i
+
\frac{\tau}{2}\mathcal M_i^k(p_i-p_H)^2
\right]  
+
\frac{\tau h}{2}
\sum_{i=1}^{N-1}
n_{i+\frac12}^k
\left|(\delta_h p)_{i+\frac12}\right|^2
\end{aligned}
\]
on the closed convex set
\[
    K_h=\{p\in\mathbb R^N:\ p_i\ge0,\ 1\le i\le N\}.
\]
Since \(n_{i+\frac12}^k\ge0\) and \(\mathcal M_i^k\ge0\), the functional
\(\mathcal J_h\) is convex. Moreover, the term
\(p_i^{(\gamma+1)/\gamma}\) is superlinear on \([0,\infty)\), so
\(\mathcal J_h\) is coercive on \(K_h\). Hence \(\mathcal J_h\) admits a
minimizer \(p^{k+1}\in K_h\).

We now show that this minimizer satisfies the finite difference scheme. Let
\[
R_i(p) = \frac{1}{h}\frac{\partial \mathcal J_h}{\partial p_i}
=
p_i^{1/\gamma}
-
n_i^k
+
\tau\mathcal M_i^k(p_i-p_H)
-
\tau(d_hF)_i .
\]
The variational inequality for the minimizer gives
\(R_i(p^{k+1})=0\) if \(p_i^{k+1}>0\), and
\(R_i(p^{k+1})\ge0\) if \(p_i^{k+1}=0\). In the latter case,
\(p_i^{k+1}\) is a minimum of the nonnegative grid function \(p^{k+1}\),
and hence \((d_hF)_i\ge0\). Therefore
\[
    R_i(p^{k+1})
    =
    -n_i^k-\tau\mathcal M_i^k p_H-\tau(d_hF)_i
    \le0 .
\]
Thus \(R_i(p^{k+1})=0\) for all \(i\). Setting
\(n_i^{k+1}=(p_i^{k+1})^{1/\gamma}\), we obtain
\[
    \frac{n_i^{k+1}-n_i^k}{\tau}
    =
    (d_hF^{k+1})_i
    -
    \mathcal M_i^k(p_i^{k+1}-p_H),
\]
which is exactly \eqref{eq:fd-scheme}. Since \(p^{k+1}\in K_h\), the
solution is nonnegative.
\end{proof}

\begin{proposition}[Homeostatic upper bound]
\label{prop:fd-upper-bound}
Assume that
\[
    0\le n_i^k\le n_H:=p_H^{1/\gamma},
    \qquad 1\le i\le N .
\]
Let \(n^{k+1}\) be a nonnegative solution of
\eqref{eq:fd-flux}--\eqref{eq:fd-reaction-mobility}. Then
\[
    0\le n_i^{k+1}\le n_H,
    \qquad
    0\le p_i^{k+1}\le p_H,
    \qquad 1\le i\le N .
\]
\end{proposition}

\begin{proof}
Assume by contradiction that \(\max_i p_i^{k+1}>p_H\), and let \(\ell\)
be such that \(p_\ell^{k+1}=\max_i p_i^{k+1}\). Then
\(n_\ell^{k+1}>n_H\) and \(p_\ell^{k+1}>p_H\). Since \(p_\ell^{k+1}\) is a maximum and
\(n_{i+\frac12}^k\ge0\), we have
\[
    (d_hF^{k+1})_\ell\le0 .
\]
Using the scheme at \(i=\ell\), we get
\[
    \frac{n_\ell^{k+1}-n_\ell^k}{\tau}
    =
    (d_hF^{k+1})_\ell
    -
    \mathcal M_\ell^k(p_\ell^{k+1}-p_H)
    \le0 .
\]
Hence \(n_\ell^{k+1}\le n_\ell^k\le n_H\), contradicting
\(n_\ell^{k+1}>n_H\). Therefore \(n_i^{k+1}\le n_H\) for all \(i\).
Together with nonnegativity and \(p_i^{k+1}=(n_i^{k+1})^\gamma\), this gives
the desired bounds.
\end{proof}

\begin{proposition}[Discrete energy dissipation]
\label{prop:fd-energy-dissipation}
Assume that \(n_i^k\ge0\) for \(1\le i\le N\). Define
\begin{equation}\label{def:modified_E}
\mathcal E_h(n)
=
h\sum_{i=1}^N
\left(
\frac{1}{\gamma+1}n_i^{\gamma+1}
-
p_Hn_i
\right).
\end{equation}
Let \(n^{k+1}\) be a nonnegative solution of
\eqref{eq:fd-flux}--\eqref{eq:fd-reaction-mobility}. Then
\begin{align}
\mathcal E_h(n^{k+1})
\le
\mathcal E_h(n^k).
\label{eq:fd-energy-dissipation}
\end{align}
\end{proposition}

\begin{proof}
Let
\[
    e(n)=\frac{1}{\gamma+1}n^{\gamma+1}-p_Hn .
\]
Since \(e\) is convex and \(e'(n)=n^\gamma-p_H\), we have
\[
    e(n_i^{k+1})-e(n_i^k)
    \le
    (p_i^{k+1}-p_H)(n_i^{k+1}-n_i^k).
\]
Multiplying \eqref{eq:fd-scheme} by
\(\tau h(p_i^{k+1}-p_H)\), summing over \(i\), and using the above
inequality, we obtain
\begin{align}
\mathcal E_h(n^{k+1})-\mathcal E_h(n^k)
&\le
\tau h
\sum_{i=1}^N
(p_i^{k+1}-p_H)(d_hF^{k+1})_i  
-
\tau h
\sum_{i=1}^N
\mathcal M_i^k|p_i^{k+1}-p_H|^2 . \label{prop:estimate1}
\end{align}
By summation by parts and the no-flux condition \eqref{eq:fd-no-flux},
\begin{align}
h\sum_{i=1}^N
(p_i^{k+1}-p_H)(d_hF^{k+1})_i
&=
-h\sum_{i=1}^{N-1}
F_{i+\frac12}^{k+1}
(\delta_hp^{k+1})_{i+\frac12} = 
-h\sum_{i=1}^{N-1}
n_{i+\frac12}^k
\left|
(\delta_hp^{k+1})_{i+\frac12}
\right|^2 . \label{prop:estimate2}
\end{align}
Here the last equality follows from the definition of \(F_{i+\frac12}^{k+1}\)
in \eqref{eq:fd-flux}.
Combining \eqref{prop:estimate1} and \eqref{prop:estimate2} yields
\begin{align}
\mathcal E_h(n^{k+1})
&+
\tau h
\sum_{i=1}^{N-1}
n_{i+\frac12}^k
\left|
(\delta_hp^{k+1})_{i+\frac12}
\right|^2
+
\tau h
\sum_{i=1}^N
\mathcal M_i^k
\left|
p_i^{k+1}-p_H
\right|^2
\le
\mathcal E_h(n^k).
\end{align}
Dropping the nonnegative
dissipation terms yields \(\mathcal E_h(n^{k+1})\le \mathcal E_h(n^{k})\).
\end{proof}

Combining Propositions~\ref{prop:fd-existence-nonnegativity},
\ref{prop:fd-upper-bound}, and \ref{prop:fd-energy-dissipation}, we immediately get the following corollary. 
\begin{corollary}[Structure-preserving properties]
\label{cor:fd-structure-preserving}
If \(0\le n_i^0\le n_H\) for \(1\le i\le N\), then for all \(k\),
\[
0\le n_i^k\le n_H,\qquad 0\le p_i^k\le p_H,
\]
and \(\mathcal E_h(n^k)\) \eqref{def:modified_E} is nonincreasing.
\end{corollary}

\subsection{Fixed-grid stiff-pressure limiting structure}
\label{subsec:large-gamma-structure}
\label{sec:ap}

We now analyze the fully discrete scheme in the stiff-pressure regime
\(\gamma\to\infty\). The mesh size \(h\) and the time step \(\tau\) are kept
fixed throughout this subsection. 

To emphasize the dependence on \(\gamma\), in this subsection we write
\[
    n_{\gamma,i}^k,\qquad
    p_{\gamma,i}^k=(n_{\gamma,i}^k)^\gamma
\]
for the solution of the fully discrete scheme
\eqref{eq:fd-scheme}--\eqref{eq:fd-reaction-mobility}. The corresponding
numerical flux and reaction mobility are denoted by
\(F_{\gamma, i+\frac12}^{k+1}\) and \(\mathcal M_{\gamma,i}^k\), respectively. 
They are defined by the same formulas as in \eqref{eq:fd-flux} and
\eqref{eq:fd-reaction-mobility}, with \(n_i^k,p_i^k\) replaced by
\(n_{\gamma,i}^k,p_{\gamma,i}^k\). The no-flux boundary condition is imposed as
before.

\begin{proposition}[Fixed-grid stiff-pressure limiting structure]
\label{prop:fixed-grid-ap-structure}
Assume that $G$ satisfies \eqref{eq:G_p_cond}-\eqref{eq:G_p_ineq}. Let \(h\), \(\tau\), and the total time steps 
\(K\ge1\) be fixed. 
Assume that, for any $\gamma>1$, the initial data satisfy
\[
    0\le p_{\gamma,i}^0\le p_H,
    \qquad 1\le i\le N,
\]
and that, up to a subsequence,
\[
    n_{\gamma,i}^{0}\to n_i^0,
    \qquad
    p_{\gamma,i}^{0}\to p_i^0,
    \qquad 1\le i\le N .
\]
Then there exists a further subsequence, still denoted by \(\gamma\), such
that, for all \(0\le k\le K\),
\[
    n_{\gamma,i}^{k}\to n_i^k,
    \qquad
    p_{\gamma,i}^{k}\to p_i^k,
    \qquad 1\le i\le N .
\]
For every \(0\le k\le K\), the limiting variables satisfy the discrete counterpart  of \eqref{eq:n_p_limit}, i.e.
\begin{equation}
    0\le n_i^{k}\le1,
    \qquad
    0\le p_i^{k}\le p_H,
    \qquad
    p_i^{k}(1-n_i^{k})=0 .
\label{eq:limit-density-pressure-complementarity}
\end{equation}
Moreover, for \(0\le k\le K-1\), they satisfy the limiting discrete density
equation
\begin{equation}
    \frac{n_i^{k+1}-n_i^{k}}{\tau}
    =
    (d_hF^{k+1})_i
    -
    \mathcal M_i^k
    \left(p_i^{k+1}-p_H\right),
    \qquad 1\le i\le N,
\label{eq:limit-density-balance}
\end{equation}
where \(F^{k+1}\) and \(\mathcal M^k\) are defined by the same formulas as
\eqref{eq:fd-flux} and \eqref{eq:fd-reaction-mobility}, with
\((n^k,p^k,n^{k+1},p^{k+1})\) replaced by the limiting variables. 
\end{proposition}

\begin{proof}
The bound-preserving property gives, uniformly in \(\gamma\),
\[
    0\le n_{\gamma,i}^{k}\le p_H^{1/\gamma},
    \qquad
    0\le p_{\gamma,i}^{k}\le p_H,
    \qquad 1\le i\le N,\quad 0\le k\le K .
\]
Since \(h\), \(\tau\), \(N\), and \(K\) are fixed, the grid-time space is
finite dimensional. Hence, up to a subsequence,
\[
    n_{\gamma,i}^{k}\to n_i^k,
    \qquad
    p_{\gamma,i}^{k}\to p_i^k,
    \qquad 1\le i\le N,\quad 0\le k\le K .
\]
Passing to the limit in the bounds gives
\[
    0\le n_i^k\le1,
    \qquad
    0\le p_i^k\le p_H .
\]


The complementarity relation follows directly from the pressure law. Since
\(p_{\gamma,i}^k=(n_{\gamma,i}^k)^\gamma\), we have
\[
p_{\gamma,i}^k(1-n_{\gamma,i}^k)
=
p_{\gamma,i}^k\left(1-(p_{\gamma,i}^k)^{1/\gamma}\right).
\]
Using the uniform bound \(0\le p_{\gamma,i}^k\le C\), one obtains
\[
\sup_{0\le a\le C}\left|a(1-a^{1/\gamma})\right|\to0,
\qquad \gamma\to\infty.
\]
Hence
\[
p_{\gamma,i}^k(1-n_{\gamma,i}^k)\to0, \quad \gamma\to\infty. 
\]
Passing to the limit gives
\[
p_i^k(1-n_i^k)=0.
\]

It remains to pass to the limit in the scheme. Since the difference
operators are finite-dimensional linear operators, we have
\[
    F_{\gamma,i+\frac12}^{k+1}
    \to
    n_{i+\frac12}^{k}
    (\delta_h p^{k+1})_{i+\frac12}
    =
    F_{i+\frac12}^{k+1}.
\]
For the mobility, using the continuous function \(g\) defined in
\eqref{eq:g-secant}, we have
\[
    \mathcal M_{\gamma,i}^k
    =
    -n_{\gamma,i}^k g(p_{\gamma,i}^k)
    \to
    -n_i^k g(p_i^k)
    =
    \mathcal M_i^k .
\]
Passing to the limit in the fully discrete scheme gives
\eqref{eq:limit-density-balance}. 
\end{proof}

\paragraph{Consistency with the Hele--Shaw limit \eqref{eq:HS_limit_pressure}.}
We next rewrite the limiting scheme in a discrete pressure-complementarity
form. Define the linearly implicit approximation of the growth rate by
\begin{equation}
G_{i,\mathrm{lin}}^k(s)
:=
g(p_i^k)(s-p_H).
\label{eq:growth-rate-linearization}
\end{equation}
This is the secant approximation of \(G(\cdot)\) anchored at \(p_H\), with
the tangent value used at \(p_i^k=p_H\). It satisfies
\[
G_{i,\mathrm{lin}}^k(p_H)=0,\qquad
G_{i,\mathrm{lin}}^k(p_i^k)=G(p_i^k),\qquad
(s-p_H)G_{i,\mathrm{lin}}^k(s)\le0 .
\]
With the notation in \eqref{eq:growth-rate-linearization}, the reaction
source in the limiting scheme is written as
\(n_i^kG_{i,\mathrm{lin}}^k(s)\). Indeed, by \eqref{eq:fd-reaction-mobility}, we have
\[
-\mathcal M_i^k(s-p_H)
=
n_i^kG_{i,\mathrm{lin}}^k(s).
\]
Then the limiting discrete density equation \eqref{eq:limit-density-balance}
becomes
\begin{equation}
\frac{n_i^{k+1}-n_i^k}{\tau}
=
(d_hF^{k+1})_i
+
n_i^kG_{i,\mathrm{lin}}^k(p_i^{k+1}).
\label{eq:limit-density-balance-source}
\end{equation}

Combining \eqref{eq:limit-density-balance-source} with the complementarity relation
\(p_i^{k+1}(1-n_i^{k+1})=0\) in \eqref{eq:limit-density-pressure-complementarity} gives 
\begin{equation}
p_i^{k+1}
\left[
\frac{1-n_i^k}{\tau}
-
(d_hF^{k+1})_i
-
n_i^kG_{i,\mathrm{lin}}^k(p_i^{k+1})
\right]
=0.
\label{eq:limit-pressure-complementarity}
\end{equation}
To identify the discrete pressure equation in saturated cells, set
\(\eta_{i+\frac12}^k:=1-n_{i+\frac12}^k\). 
Since
\[
(d_hF^{k+1})_i
=
(\delta_h^2p^{k+1})_i
-
\bigl(d_h(\eta^k\delta_hp^{k+1})\bigr)_i,
\]
we can rewrite
\eqref{eq:limit-pressure-complementarity} as
\[
p_i^{k+1}
\left[
\frac{1-n_i^k}{\tau}
-
(\delta_h^2p^{k+1})_i
+
\bigl(d_h(\eta^k\delta_hp^{k+1})\bigr)_i
-
n_i^kG_{i,\mathrm{lin}}^k(p_i^{k+1})
\right]
=0.
\]
If \(p_i^{k+1}>0\) and the neighboring old-time cells are saturated, namely
\(n_{i-1}^k=n_i^k=n_{i+1}^k=1\), then \(1-n_i^k=0\) and
\(\eta_{i-\frac12}^k=\eta_{i+\frac12}^k=0\). Hence the above relation reduces
to
\begin{equation}
(\delta_h^2p^{k+1})_i
+
G_{i,\mathrm{lin}}^k(p_i^{k+1})
=
0.
\label{eq:limit-discrete-pressure-equation}
\end{equation}
This is the finite difference Hele--Shaw-type pressure equation obtained from
the fixed-grid stiff-pressure limit. If \(G\) is linear, then
\(G_{i,\mathrm{lin}}^k(s)=G(s)\), and
\eqref{eq:limit-discrete-pressure-equation} becomes the standard finite
difference form of
\[
\Delta p+G(p)=0
\qquad \text{in } \{p>0\}.
\]
For nonlinear \(G\), \(G_{i,\mathrm{lin}}^k\) is the secant linearization
anchored at \(p_H\), with its tangent limit used when \(p_i^k=p_H\).

\paragraph{Consistency with the interface motion \eqref{eq:HS_limit_velocity}.}
We finally discuss how the limiting difference equation reflects the
free-boundary motion.  For simplicity, consider a one-dimensional
symmetric patch on the interval \((0,b)\), with symmetry at \(x=0\) and with
the right free boundary \(R(t)<b\). In the Hele--Shaw limit, the density has
the patch form
\begin{equation}
n(x,t)=
\begin{cases}
1, & 0<x<R(t),\\
0, & R(t)<x<b,
\end{cases}
\label{eq:interface-patch-density}
\end{equation}
and the pressure satisfies
\[
p(x,t)>0 \quad \text{for } 0<x<R(t),
\qquad
p(x,t)=0 \quad \text{for } R(t)<x<b .
\]
The pressure is continuous across the free boundary, so \(p(R(t),t)=0\). The
left endpoint is the symmetry point, and hence \(p_x(0,t)=0\).

In this symmetric setting, the occupied length on \((0,b)\) is exactly
\(R(t)\). Therefore the right boundary velocity can be obtained from the time
variation of the total mass
\[
R'(t)
=
\frac{d}{dt}\int_0^b n(x,t)\,dx .
\]
The same idea has a direct discrete analogue. Define the discrete total
mass on \((0,b)\) by
\[
m_h^k:=h\sum_{j=1}^{N}n_j^k .
\]
In the stiff-pressure limit, \(m_h^k\) approximates the occupied length.
Hence the mesh-scale right boundary velocity is naturally defined by
\begin{equation}
V_h^{k+1}
:=
\frac{m_h^{k+1}-m_h^k}{\tau}
=
h\sum_{j=1}^{N}\frac{n_j^{k+1}-n_j^k}{\tau}.
\label{eq:interface-discrete-velocity}
\end{equation}

Summing \eqref{eq:limit-density-balance-source} over all grid indices and
using the definition \eqref{eq:interface-discrete-velocity}, we obtain
\begin{equation}
V_h^{k+1}
=
h\sum_{j=1}^{N}(d_hF^{k+1})_j
+
h\sum_{j=1}^{N}
n_j^kG_{j,\mathrm{lin}}^k\bigl(p_j^{k+1}\bigr).
\label{eq:interface-summed-relation}
\end{equation}
Using the definition of \(d_h\) in \eqref{eq:dv_df}, we have
\[
h\sum_{j=1}^{N}(d_hF^{k+1})_j
=
F_{N+\frac{1}{2}}^{k+1}
-
F_{\frac{1}{2}}^{k+1}.
\]
The left boundary is the symmetry point, and the right computational boundary
is taken far enough with no flux. Therefore
\[
F_{\frac{1}{2}}^{k+1}
=
F_{N+\frac{1}{2}}^{k+1}
=
0.
\]
Thus \eqref{eq:interface-summed-relation} reduces to the global discrete
relation
\begin{equation}
V_h^{k+1}
=
h\sum_{j=1}^{N}
n_j^kG_{j,\mathrm{lin}}^k\bigl(p_j^{k+1}\bigr).
\label{eq:interface-global-relation}
\end{equation}

We compare \eqref{eq:interface-global-relation} with the continuous
Hele--Shaw velocity law \eqref{eq:HS_limit_velocity}. In the continuous limit, the pressure satisfies \eqref{eq:HS_limit_pressure},  the Neumann boundary condition at $x=0$ and the Dirichlet boundary condition \eqref{eq:HS_limit_boundary} at the free boundary $x=R(t)$. That is, 
\[
-p_{xx}=G(p),\qquad 0<x<R(t),
\qquad
p_x(0,t)=0,\qquad p(R(t),t)=0 .
\]
Integrating over \((0,R(t))\) gives
\[
-p_x(R(t),t)+p_x(0,t)
=
\int_0^{R(t)}G\bigl(p(x,t)\bigr)\,dx .
\]
Using \(p_x(0,t)=0\) and the velocity law in \eqref{eq:HS_limit_velocity}, we obtain
\begin{equation}
R'(t)
=
-p_x(R(t),t)
=
\int_0^{R(t)}G\bigl(p(x,t)\bigr)\,dx
=
\int_0^b n(x,t)G\bigl(p(x,t)\bigr)\,dx .
\label{eq:interface-continuous-source}
\end{equation}
Here the last equality follows from the patch structure
\eqref{eq:interface-patch-density}. Therefore the global discrete relation \eqref{eq:interface-global-relation}  is consistent with
the continuous Hele--Shaw interface identity \eqref{eq:interface-continuous-source} in this setting.


\section{Numerical experiments}
\label{sec:numerical-experiments}
\label{sec:numerics}

In this section, we test the proposed fully discrete scheme. We first test the
degenerate diffusion part and the accuracy of the growth model. We then focus
on the fixed-grid stiff-pressure limit \(\gamma\to\infty\), and finally present
two-dimensional examples with topology changes.

Unless otherwise stated, homogeneous no-flux boundary conditions are imposed. We use \(G(p)=0\) for
the pure diffusion Barenblatt benchmark, and choose
\[
G(p)=1-p,\qquad p_H=1,
\]
for all tests involving pressure-dependent growth. This choice satisfies \eqref{eq:G_p_cond} and \eqref{eq:G_p_ineq}, and the corresponding homeostatic density is \(n_H=1\).

\subsection{Accuracy and structure-preserving tests}
\label{subsec:baseline-validation}

\paragraph{Barenblatt benchmark for degenerate diffusion.}
When \(G(p)=0\), the equation \eqref{eq:tumor-model} in 1D becomes the porous medium equation
\begin{equation}
\partial_t n
=
\partial_x\bigl(n\partial_x n^\gamma\bigr)
=
\frac{\gamma}{\gamma+1}\partial_{xx}n^{\gamma+1}.
\label{eq:pme-barenblatt}
\end{equation}
This benchmark checks the degenerate diffusion part of the scheme, including
compact support, mass conservation, and energy dissipation.

Equation~\eqref{eq:pme-barenblatt} admits a time-shifted Barenblatt solution.
With \(\kappa=\gamma/(\gamma+1)\), \(\beta=1/(\gamma+2)\), and
\(s=\kappa(t+t_0)\), it is given by
\begin{equation}
n_{\rm ex}(x,t)
=
s^{-\beta}
\left(
C
-
\frac{\beta\gamma}{2(\gamma+1)}
\frac{x^2}{s^{2\beta}}
\right)_+^{1/\gamma},
\qquad
(z)_+=\max\{z,0\}.
\label{eq:barenblatt-solution}
\end{equation}
We take \(n_0(x)=n_{\rm ex}(x,0)\). For the profile comparison, we use \(\Omega=[-5,5]\),
\(h=1/64\), \(\tau=0.01\), \(T=1\), and \(t_0=2\). Two values of
\(\gamma\) are tested:
\(C=1\) for \(\gamma=3\), and \(C=0.1\) for
\(\gamma=20\), so that the support remains away from the boundary.
Figure~\ref{fig:barenblatt-profile} shows that the numerical solution agrees
well with the Barenblatt profile at \(T=1\).

\begin{figure}[htp]
  \centering
  \includegraphics[width=0.9\textwidth]{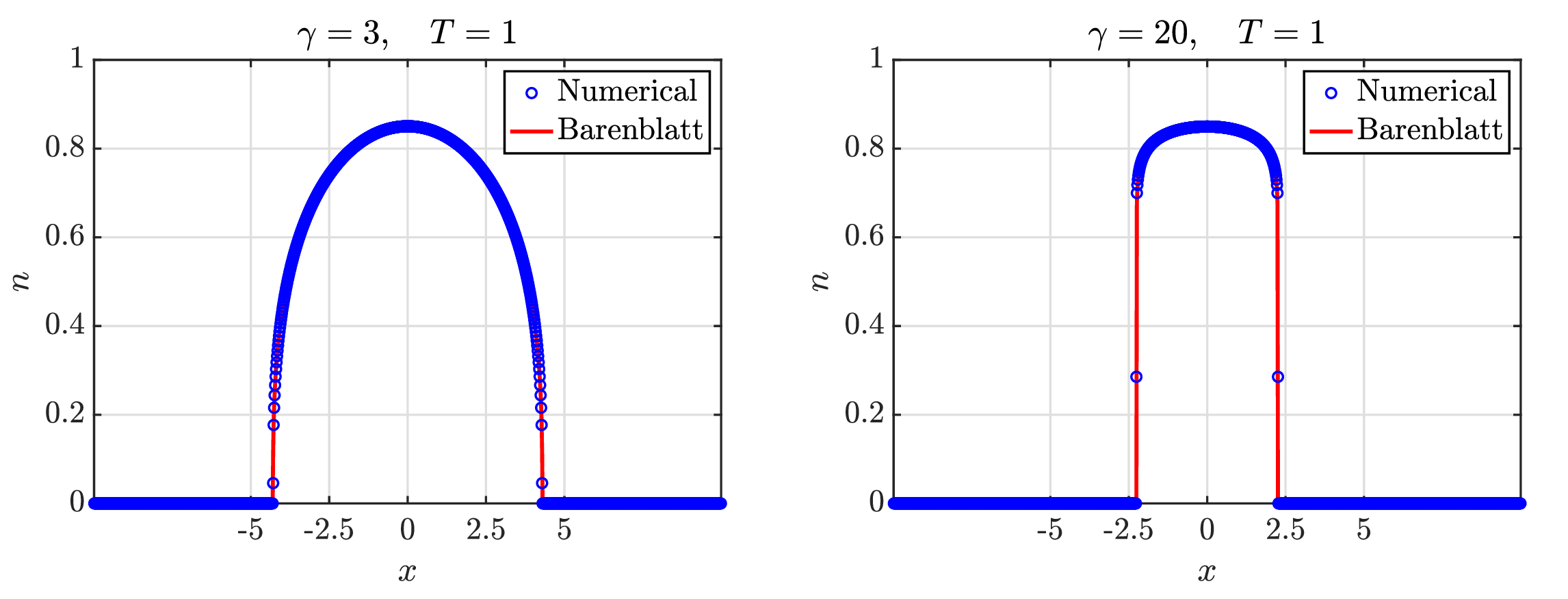}
  \caption{Comparison between the numerical solution and the Barenblatt
  profile for the porous medium equation \eqref{eq:pme-barenblatt}. Left: \(\gamma=3\), \(C=1\). Right: \(\gamma=20\), \(C=0.1\).}
  \label{fig:barenblatt-profile}
\end{figure}

Since \(G(p)=0\), the total mass is conserved and the reaction mobility
vanishes. The linear shift in the modified energy is therefore irrelevant,
and we plot the standard porous-medium energy
\[
E_h(t^k)
=
h\sum_{i=1}^{N}\frac{(n_i^k)^{\gamma+1}}{\gamma+1}.
\]
We define \(m_h(t^k):=h\sum_{i=1}^{N}n_i^k\), and monitor the relative mass
variation and the normalized energy,
\[
\frac{m_h(t^k)-m_h(0)}{m_h(0)},
\qquad
\frac{E_h(t^k)}{E_h(0)}.
\]
For the long-time conservation and dissipation diagnostics, we extend the
computational interval to \([-10,10]\) and the final time to \(T=100\), so
that the support remains away from the boundary. Figure~\ref{fig:barenblatt-structure} shows that the mass variation stays at
round-off level and that the normalized energy is nonincreasing.

\begin{figure}[htp]
  \centering
  \includegraphics[width=0.93\textwidth]{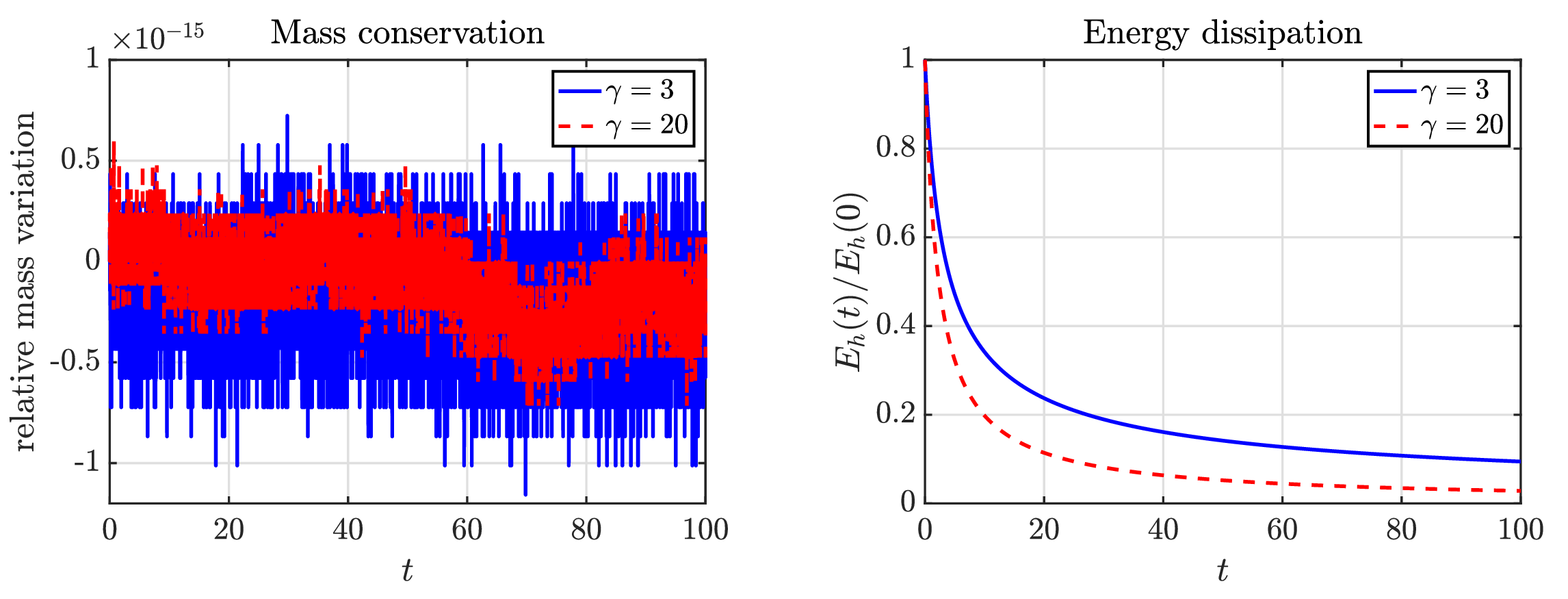}
  \caption{Mass conservation and energy dissipation for the Barenblatt
  test. Left: relative mass variation. Right: normalized energy \(E_h(t)/E_h(0)\). }
  \label{fig:barenblatt-structure}
\end{figure}

\paragraph{Accuracy test.} We now test the convergence behavior of the fully discrete scheme using the
same Barenblatt-type initial density as in the preceding benchmark. Both the
pure diffusion case and the pressure-dependent growth case specified at the
beginning of this section are considered.

Let \(n_{h,\tau,i}^k\) be the numerical density computed with mesh size
\(h\) and time step \(\tau\). The error reported in Figure~\ref{fig:error}
is defined by
\[
\mathcal E_{h,\tau}
=
\max_{k}
\left\{
h\sum_{i=1}^{N}
\left|
n_{h,\tau,i}^k
-
n_{\rm ref}(x_i,t^k)
\right|
\right\}.
\]
Thus the error is measured in the discrete \(L^1\) norm in space and then
maximized over time.

In the spatial refinement test, \(n_{\rm ref}\) is the analytical Barenblatt
solution for \(G(p)=0\) and the finest-grid numerical solution, interpolated
to the corresponding cell centers, for \(G(p)=1-p\).
In the temporal refinement test, \(n_{\rm ref}\) is computed using the
smallest time step on the same spatial grid, so that the reported error
isolates the temporal discretization error on a fixed mesh rather than the
total error to the exact solution.
We fix \(\tau=10^{-5}\) for
the spatial refinement and \(N=4096\) for the temporal refinement. The
observed results in Figure~\ref{fig:error} show a consistent decrease
in the error under both refinements, with rates close to one.

\begin{figure}[htp]
  \centering
  \includegraphics[width=0.47\textwidth]{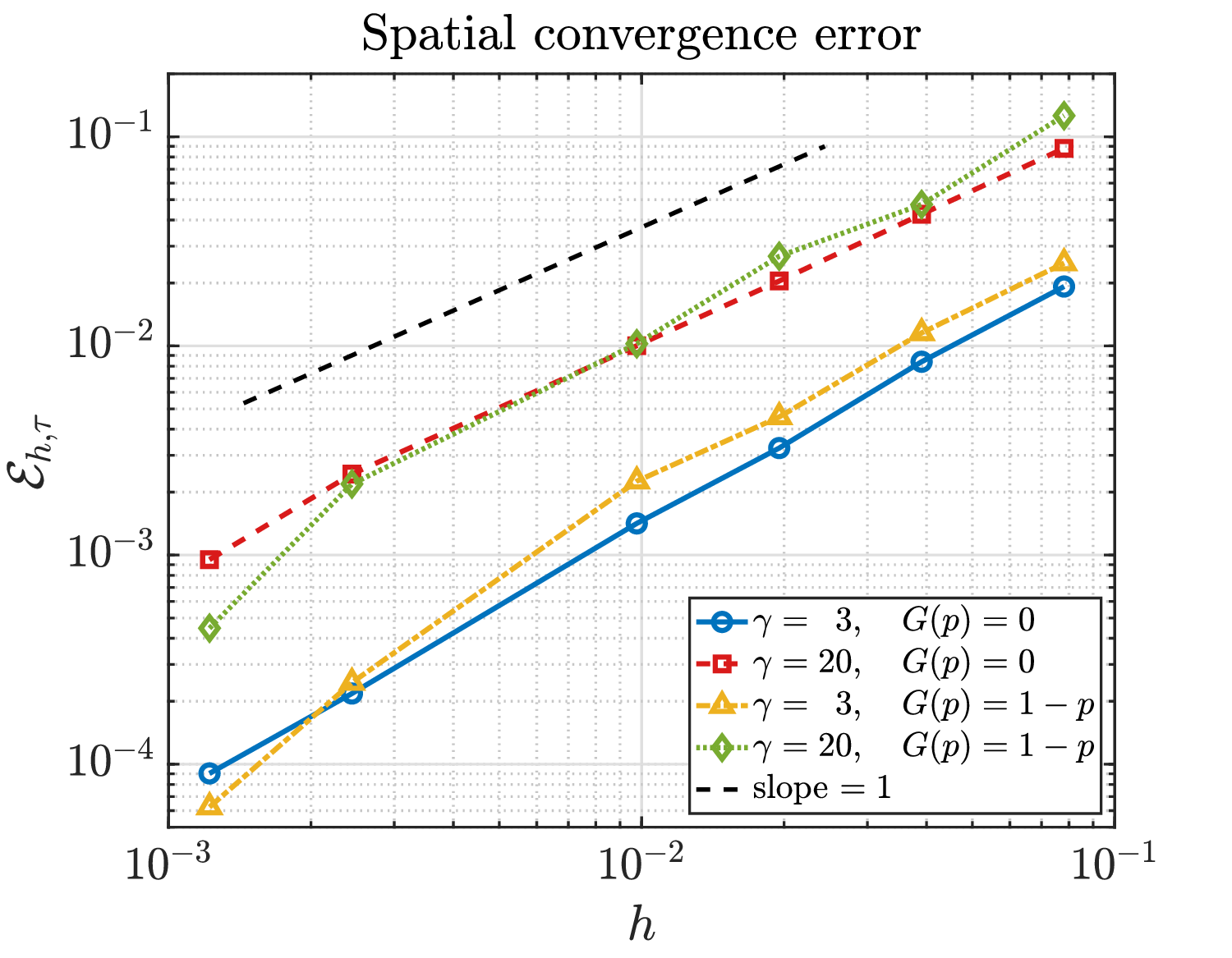}
  \includegraphics[width=0.47\textwidth]{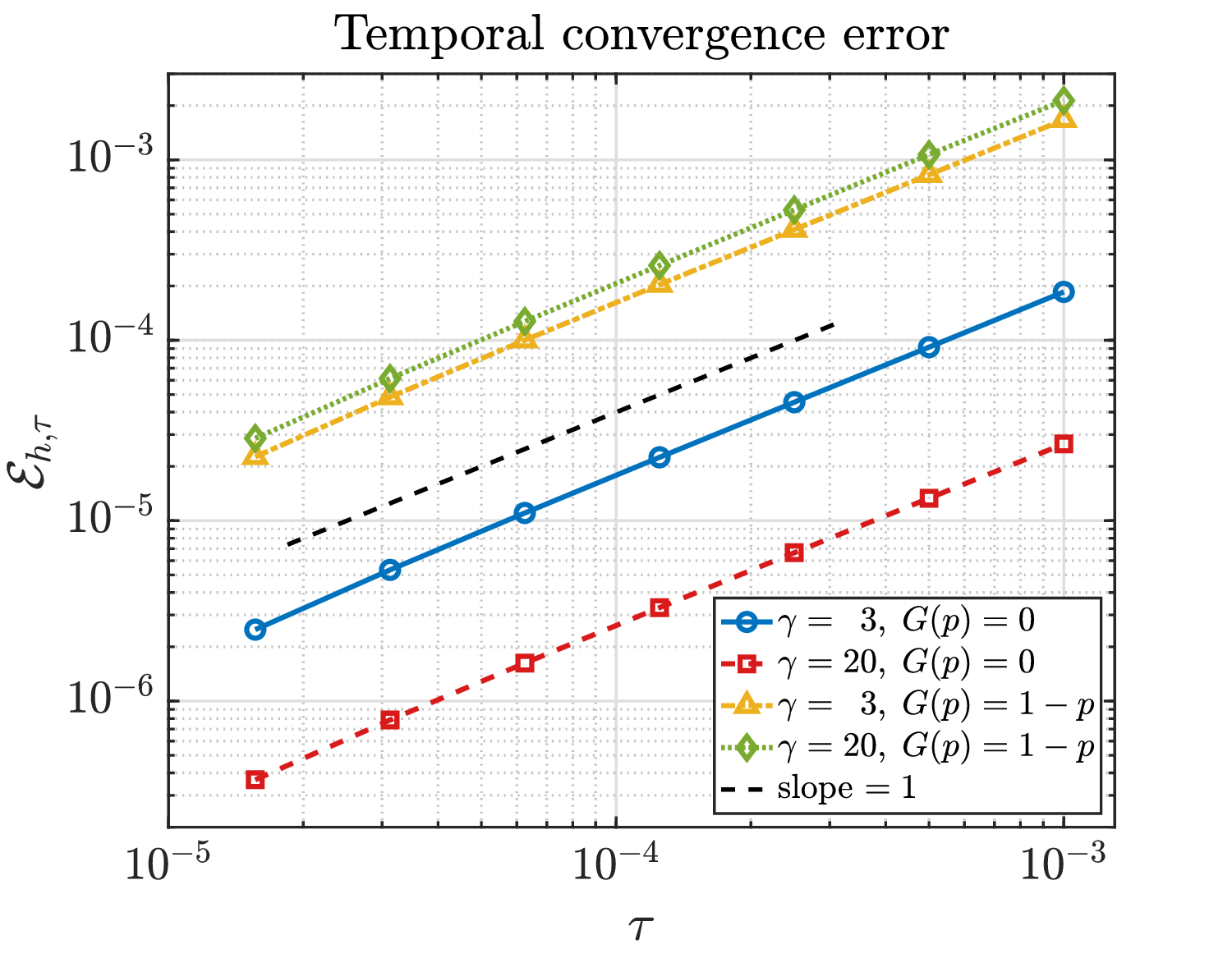}
  \caption{Convergence errors for the one-dimensional accuracy test.
  Left: spatial refinement with \(\tau=10^{-5}\). Right: temporal refinement
  with \(N=4096\). }
  \label{fig:error}
\end{figure}

\paragraph{Structure preservation for the growth model.}
We next check the discrete energy-dissipation property in the
pressure-dependent growth case \(G(p)=1-p\). Unlike the pure diffusion case,
the total mass is not conserved because of proliferation. We therefore monitor the normalized mass \(m_h(t)/m_h(0)\)
and the modified discrete energy
\[
E_h(t^k)
=
h\sum_{i=1}^N
\left(
\frac{(n_i^k)^{\gamma+1}}{\gamma+1}
-p_H n_i^k
\right).
\]
The initial density is the same as in the preceding accuracy test.  For this long-time diagnostic, we use  \(\Omega = [-20,20]\), \(h=1/64\), \(\tau=0.01\), and
run the computation up to \(T=10\). Figure~\ref{fig:general_sp} shows that the mass changes due to
proliferation, while the modified energy remains nonincreasing, in agreement
with the discrete dissipation estimate in Proposition~\ref{prop:fd-energy-dissipation}.

\begin{figure}[htp]
  \centering
  \includegraphics[width=0.95\textwidth]{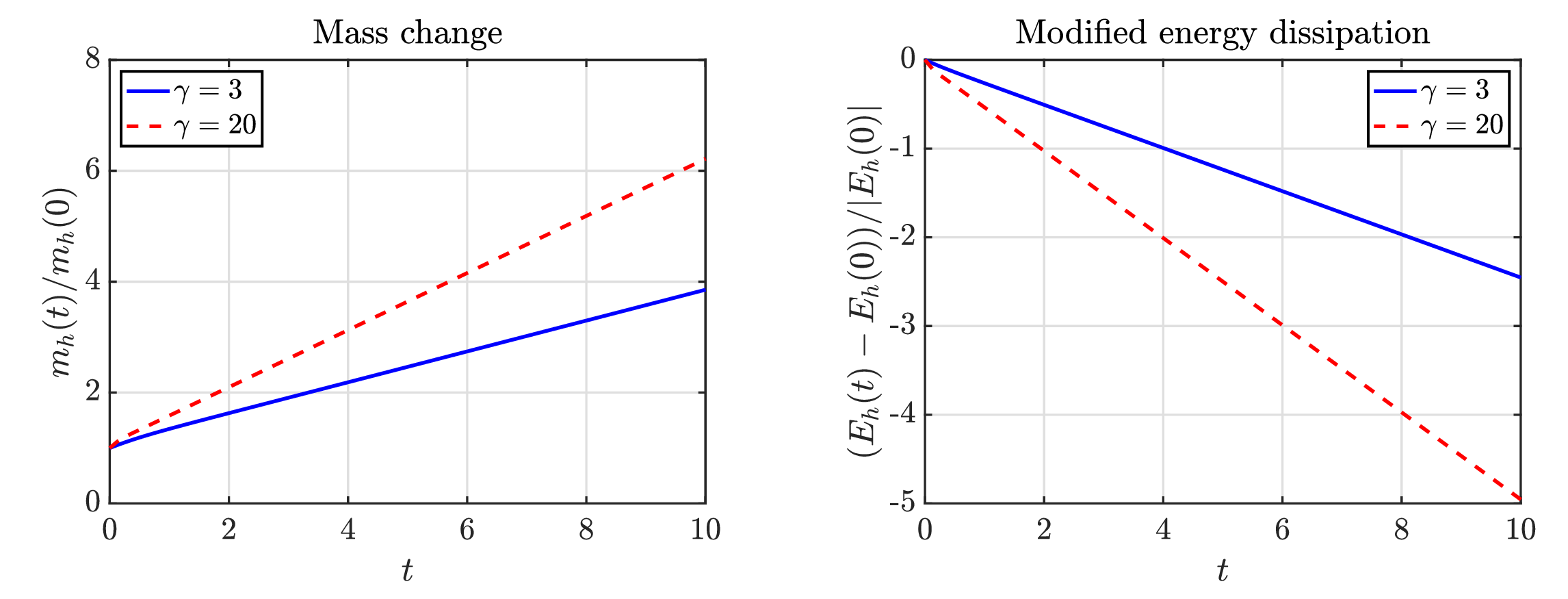}
\caption{Mass change and modified energy dissipation for the growth model
with \(G(p)=1-p\). Left: normalized mass \(m_h(t)/m_h(0)\).
Right: normalized modified energy change
\((E_h(t)-E_h(0))/|E_h(0)|\). }
\label{fig:growth-mass-energy}
  \label{fig:general_sp}
\end{figure}

\subsection{Fixed-grid stiff-pressure limit}
\label{subsec:fixed-grid-limit}

\paragraph{Hele--Shaw reference solution and setup.}
We first describe the reference Hele--Shaw solution used in the
stiff-pressure test. According to
\eqref{eq:HS_limit_pressure}--\eqref{eq:HS_limit_boundary} with
\(G(p)=1-p\), a one-dimensional symmetric patch \((-R(t),R(t))\)
satisfies
\[
-p_{\infty,xx}=1-p_\infty,\quad |x|<R(t),
\qquad
p_\infty(\pm R(t),t)=0 .
\]
By symmetry, \(p_{\infty,x}(0,t)=0\). Solving this boundary-value problem, 
together with the boundary velocity law in \eqref{eq:HS_limit_velocity}, gives
the explicit reference solution
\begin{equation}
p_\infty(x,t)
=
\left(1-\frac{\cosh x}{\cosh R(t)}\right)_+,
\quad
n_\infty(x,t)=\mathbf 1_{[-R(t),R(t)]}(x),
\quad
R(t)=\operatorname{arsinh}\bigl(e^t\sinh R_0\bigr).
\label{eq:HS-explicit-solution}
\end{equation}

For the finite-\(\gamma\) computations, with \(R_0=1\), we take
\[
p_0(x)=\left(1-\frac{\cosh x}{\cosh R_0}\right)_+,
\qquad
n_{0,\gamma}(x)=p_0(x)^{1/\gamma}.
\]
All computations in this subsection use $\Omega=[-5, 5]$, \(N=6400\), \(h=10/N\),
\(\tau=10^{-4}\), and \(T=1\). The same \(h\) and
\(\tau\) are used for all values of \(\gamma\), which gives a fixed-grid
test of the stiff-pressure limit discussed in
Section~\ref{subsec:large-gamma-structure}.

\paragraph{Convergence to the Hele--Shaw profile.}
For the profile comparison, we display the numerical solutions for
\(\gamma=3,20,200\). Figure~\ref{fig:example3-profiles} shows the density and
pressure at \(T=1\), together with the explicit Hele--Shaw profiles in
\eqref{eq:HS-explicit-solution}. As \(\gamma\) increases, the numerical density
approaches the characteristic function \(n_\infty\), while the pressure
converges to the limiting Hele--Shaw pressure. The transition in the density
becomes sharper near the free boundary, whereas the pressure profile remains
smooth and continuous.

\begin{figure}[htp]
  \centering
  \includegraphics[width=0.9\textwidth]{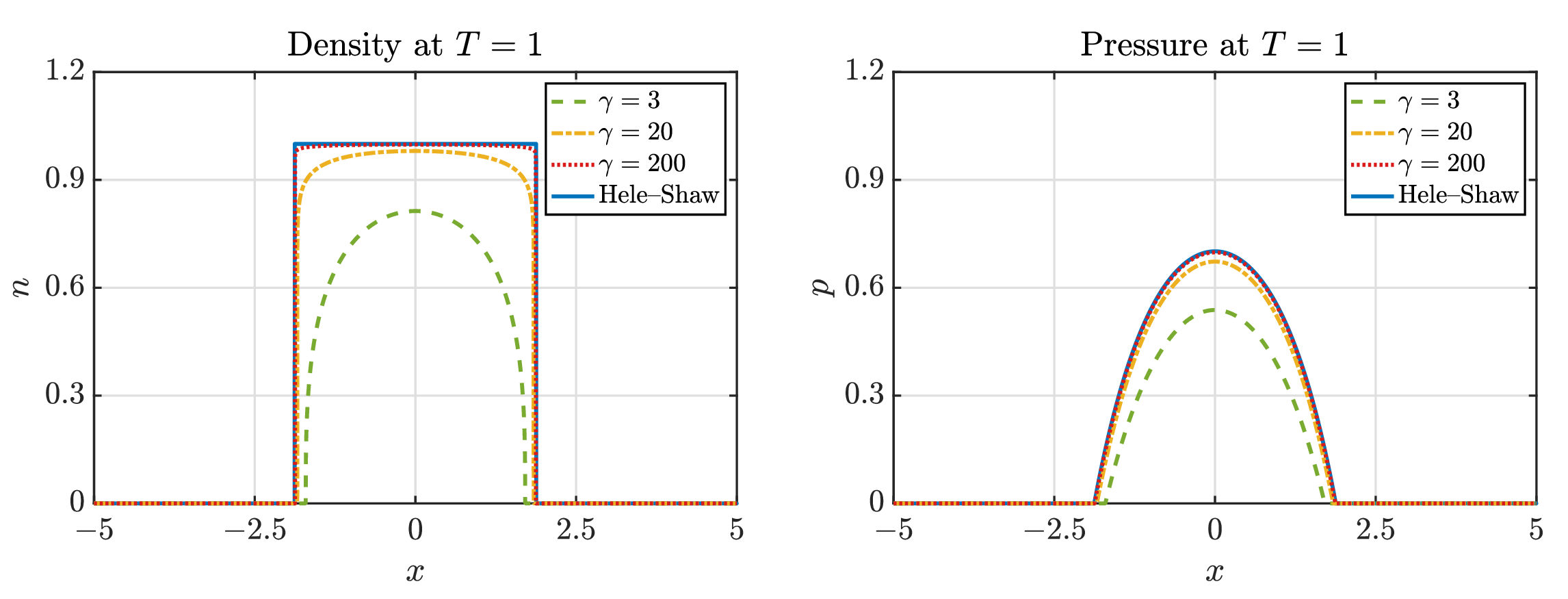}
\caption{Fixed-grid stiff-pressure limit in one dimension at \(T=1\).
Left: density profiles. Right: pressure profiles. The solid curves denote
the explicit Hele--Shaw limit.}
  \label{fig:example3-profiles}
\end{figure}

For the quantitative comparison, let \(K=T/\tau\). We first measure the
final-time \(L^1\) differences in density and pressure by
\[
e_n^\gamma(T)
=
h\sum_{i=1}^N
\left|n_{\gamma,i}^K-n_\infty(x_i,T)\right|,
\qquad
e_p^\gamma(T)
=
h\sum_{i=1}^N
\left|p_{\gamma,i}^K-p_\infty(x_i,T)\right|.
\]
For the symmetric patch limit, we also define the mass-based effective
radius and its error by
\[
R_{\gamma,h}(T)
=
\frac12 h\sum_{i=1}^N n_{\gamma,i}^K,
\qquad
e_R^\gamma(T)=|R_{\gamma,h}(T)-R(T)|.
\]
Motivated by the discrete complementarity relation
\eqref{eq:limit-density-pressure-complementarity}, we compute
\[
\mathcal C_\gamma(T)
=
h\sum_{i=1}^N
\left|p_{\gamma,i}^K\bigl(1-n_{\gamma,i}^K\bigr)\right|.
\]
To check the interior discrete Hele--Shaw pressure equation
\eqref{eq:limit-discrete-pressure-equation}, we evaluate the residual away
from the interface on
\[
\mathcal I_\eta(T)
=
\{i:\ p_\infty(x_i,T)\ge \eta\},
\qquad
\eta=0.2,
\]
and define
\[
\mathcal H_{\gamma,\eta}(T)
=
h\sum_{i\in\mathcal I_\eta(T)}
\left|
(\delta_h^2p_\gamma^K)_i+1-p_{\gamma,i}^K
\right|.
\]
Here \(\delta_h^2\) denotes the centered second difference. The final-time
diagnostics in Table~1 show that all errors decrease as \(\gamma\) increases,
confirming the convergence toward the analytical Hele--Shaw limit and the
discrete complementarity structure. The observed orders are close to one,
suggesting an approximately first-order decay with respect to \(1/\gamma\).

\begin{table}[!htbp]
\centering
\caption{Final-time diagnostics for the fixed-grid stiff-pressure test with
\(N=6400\), \(T=1\) and \(\eta=0.2\). }
\label{tab:stiff-diagnostics}
\begingroup
\small
\renewcommand{\arraystretch}{1.08}
\setlength{\tabcolsep}{5pt}
\begin{tabular}{lcccccc}
\toprule
\(\gamma\) & \(10\) & \(20\) & \(40\) & \(80\) & \(160\) & \(320\) \\
\midrule
\(e_n^\gamma(T)\)
& \(4.758{\rm E}{-1}\) & \(2.453{\rm E}{-1}\) & \(1.249{\rm E}{-1}\)
& \(6.341{\rm E}{-2}\) & \(3.231{\rm E}{-2}\) & \(1.667{\rm E}{-2}\) \\
 order
& -- & \(0.96\) & \(0.97\) & \(0.98\) & \(0.97\) & \(0.95\) \\
\hline
\(e_p^\gamma(T)\)
& \(2.271{\rm E}{-1}\) & \(1.194{\rm E}{-1}\) & \(6.119{\rm E}{-2}\)
& \(3.102{\rm E}{-2}\) & \(1.553{\rm E}{-2}\) & \(8.021{\rm E}{-3}\) \\
 order
& -- & \(0.93\) & \(0.96\) & \(0.98\) & \(1.00\) & \(0.95\) \\
\hline
\(e_R^\gamma(T)\)
& \(2.380{\rm E}{-1}\) & \(1.228{\rm E}{-1}\) & \(6.257{\rm E}{-2}\)
& \(3.181{\rm E}{-2}\) & \(1.626{\rm E}{-2}\) & \(8.439{\rm E}{-3}\) \\
 order
& -- & \(0.95\) & \(0.97\) & \(0.98\) & \(0.97\) & \(0.95\) \\
\hline
\(\mathcal C_\gamma(T)\)
& \(1.074{\rm E}{-1}\) & \(5.479{\rm E}{-2}\) & \(2.764{\rm E}{-2}\)
& \(1.388{\rm E}{-2}\) & \(6.956{\rm E}{-3}\) & \(3.480{\rm E}{-3}\) \\
 order
& -- & \(0.97\) & \(0.99\) & \(0.99\) & \(1.00\) & \(1.00\) \\
\hline
\(\mathcal H_{\gamma,\eta}(T)\)
& \(1.908{\rm E}{-1}\) & \(9.237{\rm E}{-2}\) & \(4.521{\rm E}{-2}\)
& \(2.263{\rm E}{-2}\) & \(1.090{\rm E}{-2}\) & \(4.964{\rm E}{-3}\) \\
 order
& -- & \(1.05\) & \(1.03\) & \(1.00\) & \(1.05\) & \(1.13\) \\
\bottomrule
\end{tabular}
\endgroup
\end{table}

We also record the average and maximum Newton iteration counts,
\[
\bar N_{\rm it}
=
\frac{1}{K}\sum_{k=0}^{K-1}N_k,
\qquad
N_{\rm it}^{\max}
=
\max_{0\le k<K}N_k,
\]
where \(N_k\) denotes the number of Newton iterations at the \(k\)-th time
step. The values in Table~\ref{tab:newton-gamma} show that the iteration counts remain moderate
throughout the stiff-pressure regime.

\begin{table}[!htbp]
\centering
\caption{Newton iteration counts for the fixed-grid stiff-pressure test with \(N=6400\) and \(T=1\).}
\label{tab:newton-gamma}
\begingroup
\small
\renewcommand{\arraystretch}{1.08}
\setlength{\tabcolsep}{8pt}
\begin{tabular}{lcccccc}
\toprule
\(\gamma\) & \(10\) & \(20\) & \(40\) & \(80\) & \(160\) & \(320\) \\
\midrule
\(\bar N_{\rm it}\) & \(4.71\) & \(4.58\) & \(4.52\) & \(4.43\) & \(4.38\) & \(4.36\) \\
\(N_{\rm it}^{\max}\) & \(7\) & \(6\) & \(7\) & \(8\) & \(9\) & \(9\) \\
\bottomrule
\end{tabular}
\endgroup
\end{table}

\subsection{Two-dimensional free boundary evolutions}
\label{subsec:2d-free-boundary}

We finally present two-dimensional simulations to illustrate free-boundary
evolutions with topology changes. The simulations use the two-dimensional
extension described in Appendix~\ref{app:2d-extension}. Both tests use
\(500\times500\) uniform cells on \([-5,5]^2\), so that \(h_x=h_y=0.02\),
with \(\gamma=40\) and \(\tau=10^{-3}\). The level set \(n=0.5\) is used only to visualize the
numerical free boundary and is not used by the scheme. In
Figures~\ref{fig:annulus-topview}--\ref{fig:twodisks-topview}, only the
central region \([-4,4]^2\) is displayed.

\paragraph{Hole filling from an annular initial density.}
The first example starts from an annular density,
\[
n_0(x,y)
=
0.8\,\mathbf 1_{\{r_{\rm in}\le \sqrt{x^2+y^2}\le r_{\rm out}\}}(x,y),
\qquad
r_{\rm in}=0.8,\quad
r_{\rm out}=2.0 .
\]
Figure~\ref{fig:annulus-topview} shows
the expansion of the outer boundary and the contraction of the inner boundary,
leading to the transition from a doubly connected region to a simply connected
region.

\begin{figure}[htp]
  \centering
  \includegraphics[width=0.97\textwidth]{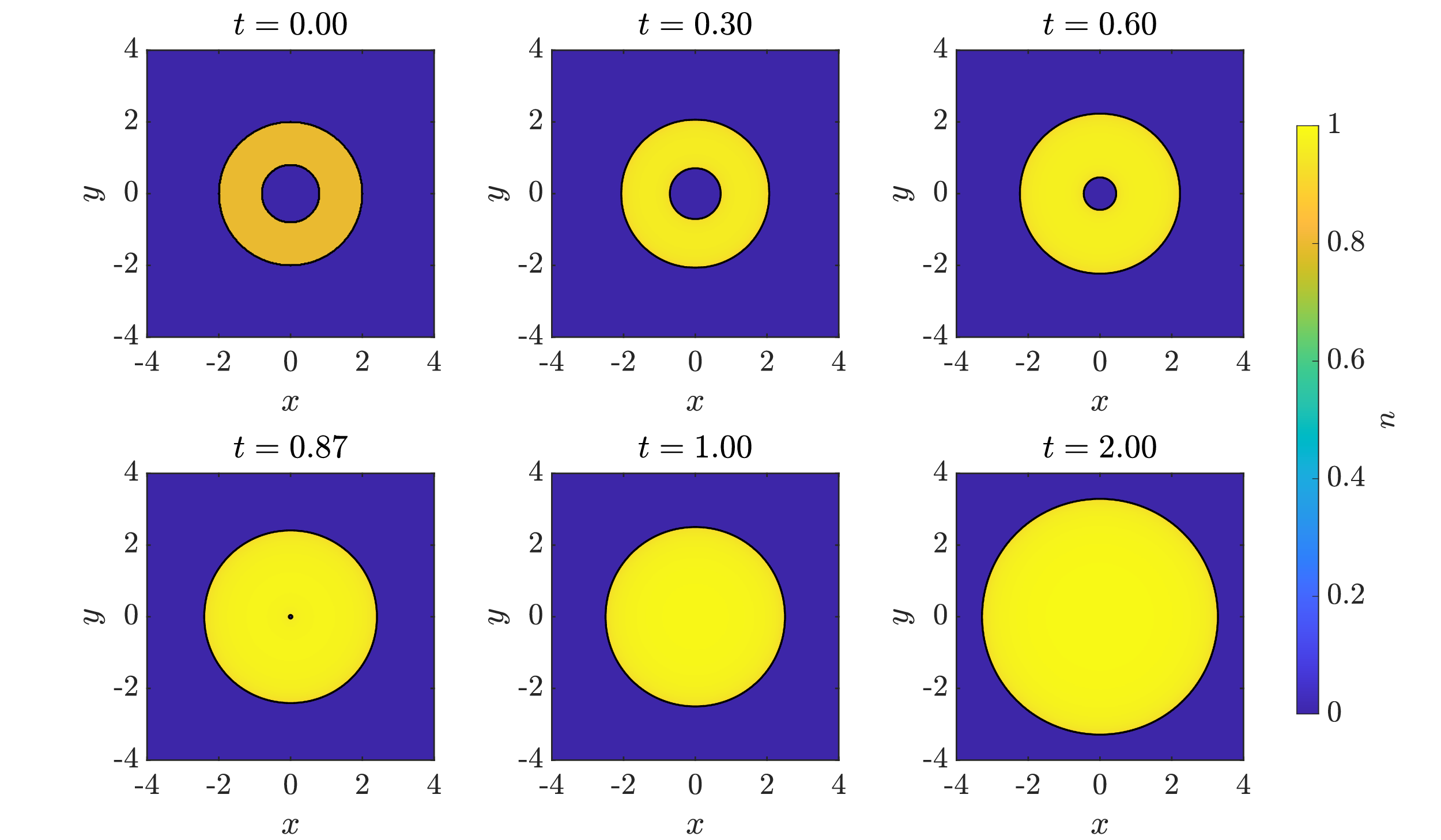}
  \vspace{-5mm}
 \caption{Density snapshots for hole filling from an annular initial density.}
  \label{fig:annulus-topview}
\end{figure}

\paragraph{Merging of two tumor patches.}
The second example starts from two separated disks,
\[
n_0(x,y)
=
0.8\,\mathbf 1_{B((-c_0,0),r_0)\cup B((c_0,0),r_0)}(x,y),
\qquad
r_0=0.8,\quad
c_0=1.4 .
\]
Here \(B(z,r)\) denotes the disk centered at \(z\) with radius \(r\).
Figure~\ref{fig:twodisks-topview} shows that the two components first expand
separately, then touch each other, and finally merge into a single connected
region. After \(t=3\), the merged patch continues to expand without further
topology change.

\begin{figure}[htp]
  \centering
  \includegraphics[width=0.97\textwidth]{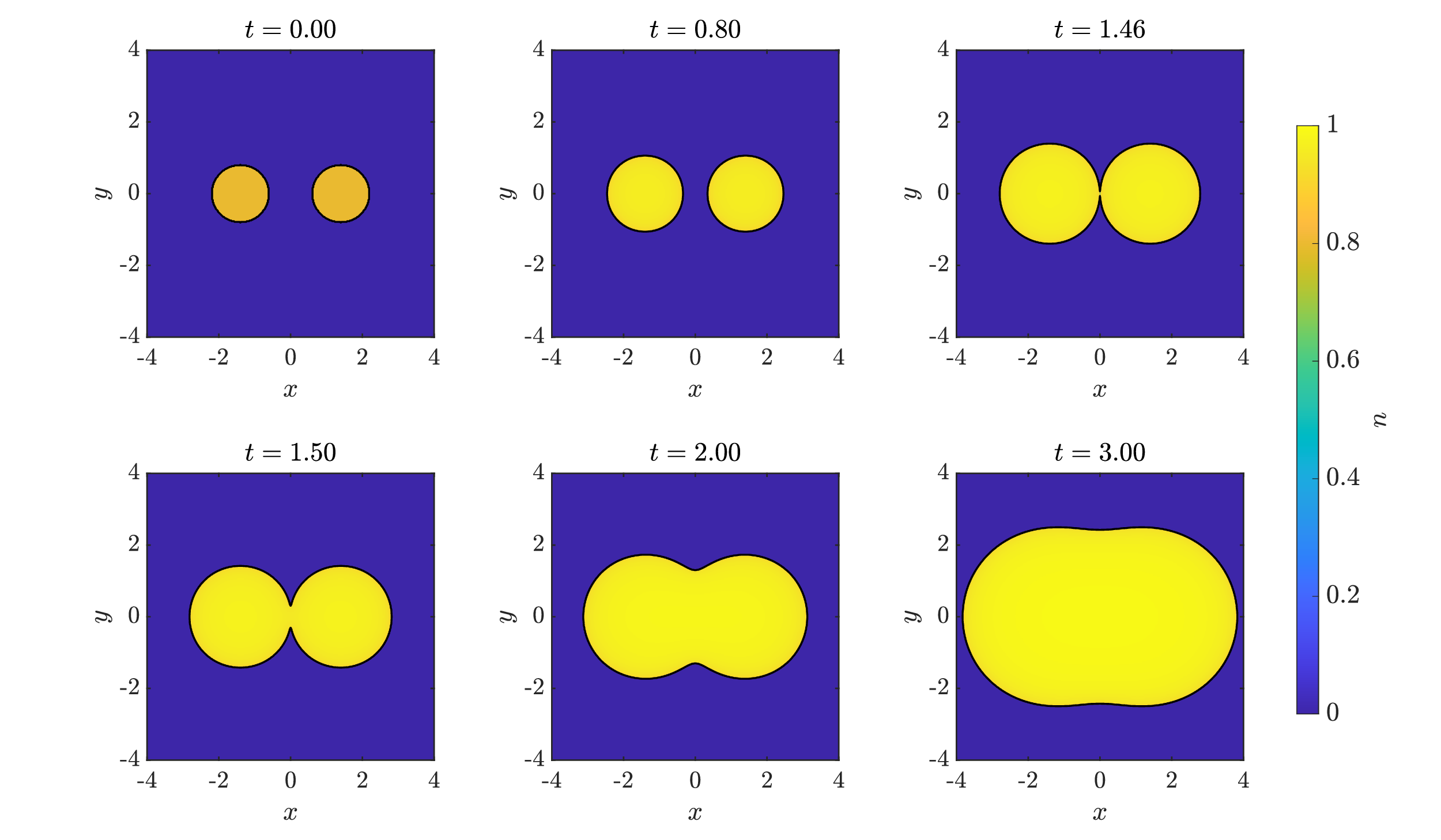}
  \vspace{-5mm}
\caption{Density snapshots for the merging of two initially separated tumor patches.}
  \label{fig:twodisks-topview}
\end{figure}

These two examples show that the fixed-grid scheme captures hole filling and
component merging without explicit interface tracking.

\section{Conclusion}
\label{sec:conclusion}

We have developed a structure-preserving finite difference method for a
pressure-driven tumor growth model with pressure-dependent proliferation.
The scheme is derived from the Onsager variational principle using a modified
energy shifted by the homeostatic pressure, which makes the reaction term
dissipative and allows transport and growth to be treated in a unified
Rayleighian formulation. The resulting fully discrete scheme uses explicit
mobilities and an implicit pressure update. We proved nonnegativity
preservation, the homeostatic upper bound, discrete modified energy
dissipation, and a fixed-grid stiff-pressure limiting structure.

The numerical results confirm the accuracy and structure-preserving behavior
of the method. The
scheme captures Barenblatt profiles in the pure diffusion case, preserves
mass in the no-growth case, dissipates the appropriate energy, and converges
toward the Hele--Shaw limit as \(\gamma\) becomes large. Two-dimensional tests
show that the method can handle free boundary evolutions with topology
changes, including hole filling and the merging of separated patches. Future
work will consider extensions to models with nutrients, chemotaxis, multiple
cell populations, and adaptive resolution near free boundaries.

\bigskip
\noindent{\large\bf Acknowledgements.} X. Ruan acknowledges support  from the National Natural Science Foundation of China under grant  12201436 and the R\&D Program of Beijing Municipal Education Commission under grant KM202310028016.
W. Huang acknowledges support  from the National Natural Science Foundation of China under grant  12001034.

\appendix

\section{Two-dimensional extension}
\label{app:2d-extension}

In two space dimensions, the scheme is extended by using the same
conservative finite-difference discretization in the \(x\)- and
\(y\)-directions. Let \(h_x\) and \(h_y\) be the mesh sizes. We set
\[
p_{i,j}^k=(n_{i,j}^k)^\gamma,
\qquad
M_{i,j}^k=-n_{i,j}^k g(p_{i,j}^k).
\]
The face mobilities are evaluated explicitly by arithmetic averages,
\[
n_{i+\frac12,j}^k
=
\frac{n_{i,j}^k+n_{i+1,j}^k}{2},
\qquad
n_{i,j+\frac12}^k
=
\frac{n_{i,j}^k+n_{i,j+1}^k}{2}.
\]
The pressure-gradient fluxes are
\[
F_{i+\frac12,j}^{x,k+1}
=
n_{i+\frac12,j}^k
\frac{p_{i+1,j}^{k+1}-p_{i,j}^{k+1}}{h_x},
\qquad
F_{i,j+\frac12}^{y,k+1}
=
n_{i,j+\frac12}^k
\frac{p_{i,j+1}^{k+1}-p_{i,j}^{k+1}}{h_y}.
\]
With zero normal flux imposed on the boundary, the update is
\[
\frac{n_{i,j}^{k+1}-n_{i,j}^k}{\tau}
=
(d_xF^{x,k+1})_{i,j}
+
(d_yF^{y,k+1})_{i,j}
-
M_{i,j}^k
\left(p_{i,j}^{k+1}-p_H\right),
\qquad
p_{i,j}^{k+1}=(n_{i,j}^{k+1})^\gamma ,
\]
where
\[
(d_xF^{x,k+1})_{i,j}
=
\frac{F^{x,k+1}_{i+\frac12,j}-F^{x,k+1}_{i-\frac12,j}}{h_x},
\qquad
(d_yF^{y,k+1})_{i,j}
=
\frac{F^{y,k+1}_{i,j+\frac12}-F^{y,k+1}_{i,j-\frac12}}{h_y}.
\]
The positivity preservation, homeostatic upper bound, and modified energy
dissipation are obtained by the same arguments as in one dimension, using
summation by parts in the two coordinate directions.

\section{Implementation of the nonlinear solver}
\label{app:nonlinear-solver}

This appendix summarizes the nonlinear solver used for the fully discrete
scheme. At each time step, the coefficients \(n^k_{i+\frac12}\) and
\(\mathcal M_i^k\) are evaluated explicitly from the known solution, with
\(\mathcal M_i^k=-n_i^kg(p_i^k)\). The nonlinear unknown is
\(u=\{u_i\}_{i=1}^N\), representing \(n^{k+1}\), and the pressure is
\(p_i(u)=u_i^\gamma\). The fully discrete scheme is written as
\[
R_i(u):=
u_i-n_i^k+\tau \mathcal M_i^k(u_i^\gamma-p_H)-\tau(d_hF(u))_i=0,
\qquad 1\le i\le N,
\]
where
\[
F_{i+\frac12}(u)
=
n^k_{i+\frac12}\frac{u_{i+1}^\gamma-u_i^\gamma}{h},
\qquad
F_{\frac12}(u)=F_{N+\frac12}(u)=0 .
\]

We solve \(R(u)=0\) by a damped Newton method. Given \(u^{(\ell)}\), the
Newton correction \(s^{(\ell)}\) is obtained from
\[
J_R(u^{(\ell)})s^{(\ell)}=-R(u^{(\ell)}),
\]
where \(J_R(u)\) is the Jacobian matrix of the residual. The Jacobian is
tridiagonal in one space dimension and has the standard five-point sparse
structure in two space dimensions. A backtracking line search chooses
\(\alpha_\ell\in(0,1]\) so that \(u^{(\ell)}+\alpha_\ell s^{(\ell)}\ge0\)
and the residual decreases sufficiently. The new iterate is
\[
u^{(\ell+1)}=u^{(\ell)}+\alpha_\ell s^{(\ell)} .
\]
The iteration is stopped when
\[
\frac{\|R(u^{(\ell)})\|_\infty}{1+\|n^k\|_\infty}
\le \mathrm{tol}.
\]
In the computations, we use \(u^{(0)}=n^k\) and \(\mathrm{tol}=10^{-10}\).
The one-dimensional Newton systems are solved by a direct tridiagonal solver,
while the two-dimensional systems are solved by a sparse linear solver.

\end{document}